\documentclass[11pt]{amsart}
\usepackage{amssymb}
\usepackage{amsthm}
\usepackage{amsmath}
\usepackage[colorlinks=true]{hyperref}
\usepackage{graphicx}
\usepackage[all]{xy}

\DeclareMathOperator{\crit}{crit}
\DeclareMathOperator{\Rp}{Re}

\newcommand{\R}{\mathbb{R}}
\newcommand{\C}{\mathbb{C}}
\newcommand{\Z}{\mathbb{Z}}

\theoremstyle{definition}
\newtheorem{lemma}{Lemma}
\newtheorem{ex}{Example}
\newtheorem{rmk}{Remark}
\newtheorem{prop}{Proposition}
\newtheorem{cor}{Corollary}
\newtheorem{df}{Definition}

\newtheorem{thm}{Theorem}
\newtheorem{conj}{Conjecture}

\begin{document}

\baselineskip.5cm
\title{The $h$-principle for broken Lefschetz fibrations}
\author[Jonathan Williams]{Jonathan Williams}
\address{Department of Mathematics, The University of Texas at Austin \newline
\hspace*{.375in}Austin, Texas 78712}
\email{\rm{jwilliam@math.utexas.edu}}
\thanks{The author would like to thank the following people, for without their kind support and thoughtful comments this work would not have appeared: Denis Auroux, R. \.{I}nan\c{c} Baykur, John Etnyre, David Gay, my thesis advisor Bob Gompf, Maxim Kazarian, Yank\i\ Lekili, and Rustam Sadykov.}
\begin{abstract}It is known that an arbitrary smooth, oriented 4-manifold admits the structure of what is called a \emph{broken Lefschetz fibration}. Given a broken fibration, there are certain modifications, realized as homotopies of the fibration map, that enable one to construct infinitely many distinct fibrations of the same manifold. The aim of this paper is to prove that these modifications are sufficient to obtain every broken fibration in a given homotopy class of smooth maps. One notable application is that adding an additional ``projection" move generates all broken fibrations, regardless of homotopy class. The paper ends with further applications and open problems.\end{abstract}
\maketitle
\tableofcontents
\begin{section}{Introduction}Over the past fifteen years, the interplay between symplectic topology and gauge theory has resulted in significant progress in the understanding of smooth 4-manifolds. More recent developments have resulted in an extension of this understanding toward the nonsymplectic setting via new formulations of smooth invariants, for example the standard surface count of \cite{DS}, the Heegaard-Floer theory of \cite{OS1} and the Lagrangian matching invariant of \cite{P1}. While the first was shown to equal the Seiberg-Witten invariant in \cite{U} in certain ``preferred" situations, the latter two are still only conjecturally equivalent to it. Comparisons and relationships between the three invariants abound. Relevant to the focus of this paper, they all involve a choice of fibration or handlebody structure (which are sometimes interchangeable) on which the corresponding invariant is presumed not to depend. The Heegaard-Floer invariant for 4-manifolds begins with a handlebody decomposition of a twice-punctured 4-manifold, which induces maps between the three-dimensional invariants of its boundary components. The Lagrangian matching invariant requires a \emph{broken Lefschetz fibration}, which in some ways resembles a circle-valued morse function on a 3-manifold without extrema, except with an extra dimension in both the source manifold and target sphere. This paper begins by presenting a list of modifications (which first appeared in \cite{L}) that may be performed on any broken fibration to produce a new one, and that can be used to modify any broken fibration such that it induces a handlebody decomposition of the 4-manifold in a straightforward way. Connected to the currently open question of whether the Lagrangian matching invariant is indeed independent of the chosen fibration structure, these modifications present a topological question which is interesting in its own right: are they \emph{complete} in the sense that they generate the entire collection of broken fibrations in a fixed homotopy class of maps? The main result of this paper is that they are indeed complete; that is, for a given homotopy class of maps, broken Lefschetz fibrations are unique up to these modifications.\end{section}

\begin{section}{A calculus of stable maps}\label{calc}Modifying the critical loci of broken Lefschetz fibrations (and thus possibly the diffeomorphism type of the total space) goes as far back as the first existence result, appearing in \cite{ADK}. In this paper there is a fixed smooth, compact 4-manifold, denoted $M$ throughout, whose spatial coordinates are given singly by $x_i$, $i=1,\ldots,4$ and collectively by $x$. For any homotopy, the homotopy parameter is always denoted $t$, and often it will be necessary to view a homotopy both as a 1-parameter family of maps and as a map $I\times M\rightarrow I\times F$ that respects the product structure of both spaces; thus, in order to streamline notation, subscripts will denote products when there is no ambiguity; for example, the notation $M_t$ will denote the slice $\{t\}\times M\subset[0,1]\times M$, and when appropriate, the symbol $M_I$ will denote $I\times M$, $I=[0,1]$. Using an appropriate notion of stability, the critical locus of a $k$-parameter family of stable maps between low-dimensional manifolds was characterized as part of a framework of local models in a rather general paper appearing in 1975 \cite{W}. In the special case of a 1-parameter family of maps $M^4\rightarrow S^2$ without definite folds, a concise and accessible description of the fibration structure imposed by these critical points appeared in \cite{L}. Wassermann's stability criterion allows for finitely many points $(t,m)\in M_I$ at which the map $M_{t}\rightarrow S^2_{t}$ fails to be a stable map, with three explicit local models for those points. These local models have descriptions as local modifications of an existing stable map $D^4\rightarrow D^2$ by homotopy. Extending this list to encompass definite folds, this section is a description of the tools used in the proof of the main theorem and the objects to which they apply. Beginning with \cite{ADK}, broken fibrations have been depicted and studied by drawing pictures of the base of the fibration along with the critical values of the map, adding decorations that describe the behavior of the critical locus with respect to the fibration structure adjacent to it (in this paper they are called \emph{base diagrams}). While this approach has its limitations as noted in Section \ref{appendix}, it presents a useful medium to introduce and study broken fibrations and stable maps in a general sense. It is useful to understand the critical locus of a family of such maps in more than one light: the first description is from the familiar perspective that it is a map from a 5-manifold to a 3-manifold such that the points at which its derivative vanishes obey certain specific local models.
\begin{subsubsection}{Stability}The main object of study in this paper is a 1-parameter family of maps $f_t:M^4\rightarrow S^2$, $t\in[0,1]$, endowed with a notion of stability such that, for all but finitely many values of $t$, $f_t$ is itself a stable map, and that moreover $f_t$ as a map from a 5-manifold to a 3-manifold is stable within the class of 1-parameter families. Given an appropriate equivalence relation for maps, $(r,s)$-stability is the condition that any sufficiently small $r$-parameter perturbation of a map which is itself an $s$-parameter family of maps preserves the equivalence class of that map within the collection of $s$-parameter families. In this paper, $r=s=1$; hence the term $(1,1)$-stability. For the bare notion of stability of maps between manifolds, that relation is \emph{right-left} equivalence, recognizable to topologists as simply change of coordinates, where two maps $g_1,g_2:X\rightarrow Y$ are equivalent when there are diffeomorphisms $f,h$ such that $g_1f=hg_2$. Then $f$ is stable when for each perturbation $f_t$ such that $f_0=f$, there exists some $\epsilon>0$ such that $f_t$ is equivalent to $f$ when $|t|<\epsilon$. More appropriate for this discussion is an equivalence that respects the product structure associated to a 1-parameter family of maps.
\begin{df}Let $f,g:\R^6\rightarrow\R$ be map germs with $f(0)=g(0)=0$. Associated to $f$ and $g$ are germs $F:\R^6\rightarrow\R^3$, $G:\R^6\rightarrow\R^3$, defined by $F(s,t,x_1,x_2,x_3,x_4)=(s,t,f(s,t,x_1,x_2,x_3,x_4))$ and $G(s,t,x_1,x_2,x_3,x_4)=(s,t,g(s,t,x_1,x_2,x_3,x_4))$, respectively. Then $f$ and $g$ are \emph{$(1,1)$-equivalent} if there are germs $\Phi,\Lambda,\psi,\phi$ of diffeomorphisms fixing the origin such that the following diagram commutes:\begin{equation}\label{eq:rsequiv}\xymatrix{\R^6\ar[d]^{\Phi}\ar[r]^{F} & \R^3\ar[d]^{\Lambda}\ar[r]^{p} & \R^2\ar[d]^{\psi}\ar[r]^q & \R\ar[d]^\phi\\\R^6\ar[r]^G & \R^3\ar[r]^p & \R^2\ar[r]^q & \R}\end{equation}where $p:\R^2\times\R\rightarrow\R^2$ and $q:\R\times\R\rightarrow\R$ are projections onto the first factor.\end{df}
As a map $\R^n\rightarrow\R^p$ is simply a $p$-tuple of maps $\R^n\rightarrow\R$, this definition easily generalizes to higher-dimensional target spaces (for instance, two dimensions is important for this paper, as we are considering maps from a 4-manifold to the 2-sphere). A slight generalization eliminates the condition that the maps be equivalent at the origin:
\begin{df}Let $\mathcal{E}(n)$ denote the space of map germs $\R^n\rightarrow\R$ represented by functions that fix the origin, and let $U$ be an open subset of $\R^n$. Let $f:U\rightarrow\R$ be a smooth function and let $z\in U$. Define the germ $f_z\in\mathcal{E}(p)$ by setting $f_z(y)=f(z+y)-f(z)$ for all $y$ near $0\in \R^n$. Let $U$ and $V$ be open subsets of $\R^{n+r+s}$ and let $f:U\rightarrow\R$, $g:V\rightarrow\R$ be smooth functions. Let $(x,u,v)\in U$ and let $(y,w,t)\in V$. Then $f$ at $(x,u,v)$ is \emph{$(r,s)$-equivalent to $g$ at $(y,w,t)$} if the germs $f_{(x,u,v)}$ and $g_{(y,w,t)}$ in $\mathcal{E}(n+r+s)$ are $(r,s)$-equivalent.\end{df}
\begin{df}Let $f\in\mathcal{E}(n+1+1)$. Then $f$ is \emph{weakly $(1,1)$-stable} if for every open neighborhood $U$ of $0$ in $\R^{n+1+1}$ and for every representative function $f':U\rightarrow\R$, the following holds: For any smooth function $h:U\rightarrow\R$ there is a real number $\epsilon>0$ such that if $t$ is any real number with $|t|<\epsilon$, then there is a point $(x,u,v)\in U$ such that $f'+th$ at $(x,u,v)$ is $(1,1)$-equivalent to $f'$ at $0$.\end{df}According to Theorem 3.15 of \cite{W}, the $(1,1)$ stability that appears in \cite{L} is equivalent to the one above; this one was chosen because it is more intuitive to the author both on its own and as a generalization of stability. This paper refers to a weakly $(1,1)$-stable family of maps from a smooth compact 4-manifold to a surface $F$ as a \emph{deformation} for short.\end{subsubsection}
\begin{subsection}{The Thom-Boardman Stratification}For a smooth map $f$ from an $n$-dimensional manifold $N$ to a $p$-dimensional manifold $P$, a critical point is simply any point $x\in N$ such that the derivative $df:TN\rightarrow TP$ satisfies $rk(df_x)<p$. Denoting the \emph{kernel rank} of $f$ at $x$ as $kr_f(x)=p-rk(df_x)$, the critical locus $S(f)$ is defined to be $\{x\in N:kr_f(x)>0\}$. Keeping track of kernel dimension, a common notation $S^{(k)}f$ denotes the locus of points where $kr_f(x)\leq k$; in other words, $S^{(k)}f$ is the closure of the locus of points where $f$ \emph{drops rank} by at least $k$ dimensions. In order to denote the locus of points where $kr_f$ is precisely $k$, one writes $S^{(k,0)}f$.\\

For the arbitrary smooth map $f:N\rightarrow P$, $S^{(k)}f$ is generally not a submanifold of $N$; for example, $S^{(3)}$ for the midpoint $t=0$ of the merging homotopy of Figure \ref{merge} is a pair of arcs which meet at a higher-order critical point in the interior of each. Supposing $f$ is stable, it is known that $S^{(k)}f$ is a smooth submanifold of $X$ which necessarily has positive codimension for $k\neq0$. Consequently, the restriction $f'=f|_{S^{(k)}f}$ is itself a smooth map between smooth manifolds, and there is a smooth submanifold of $S^{(k)}f$ defined by $S^{(\ell)}(f')$, which is denoted $S^{(k,\ell)}f$. Inductively, a \emph{Boardman symbol} is a nonincreasing sequence of positive integers $\mathcal{I}=(k_1,k_2,\ldots,k_n)$, where the entries of the sequence denote the kernel ranks of successive restrictions of $f$. Following the pattern, the corresponding stratum in the critical locus is denoted $S^\mathcal{I}f$.\end{subsection}
\begin{prop}\label{defsarestable}A deformation $f:M_I\rightarrow S^2_I$ is stable as a map from a 5-manifold to a 3-manifold. Consequently, there is a submanifold $S^\mathcal{I}f$ for each symbol $\mathcal{I}$.\end{prop}\begin{proof}The commutative square appearing in Equation \ref{eq:rsequiv} specializes to right-left equivalence by considering the maps between the four outermost corners as follows:\begin{equation}\xymatrix{\R^6\ar[d]^{\Phi}\ar[r]^{p\circ q\circ F} & \R\ar[d]^\phi\\\R^6\ar[r]^{p\circ q\circ G} & \R}\end{equation}From there, the right-left equivalence after sufficiently small perturbation follows through the rest of the definitions with little difficulty.\end{proof}
\begin{subsection}{Critical loci of deformations}Viewing $\R^4$ with complex coordinates $(z,w)$, various audiences will recognize the following map as the local model for a Lefschetz critical point:\begin{equation}\label{eq:lcp}(z,w)\mapsto zw.\end{equation}A smooth map whose critical locus consists of isolated Lefschetz critical points is called a \emph{Lefschetz fibration}. The local and global properties are discussed in detail in \cite{GS}. There are several other classical, locally defined critical point models, which are of interest in this paper because they characterize what it means to be a deformation. These critical points also impose certain fibration-like structures on $M$, which are naturally described in the order given by their stratification in the next few paragraphs. Though the following discussion involves deformations as the relevant objects of interest, much of what follows is also true for stable maps from 4-manifolds to surfaces and from 5-manifolds to 3-manifolds.
\begin{subsubsection}{Folds}The highest stratum, indeed the entire critical locus of a deformation, is $S^{(3)}$, the closure of the fold locus. For a map from a 4-manifold $M$ to a surface, it is defined by the following local model $\R^4\rightarrow\R^2$ up to change of local coordinates:\begin{equation}\label{eq:fold1}(x_1,x_2,x_3,x_4)\mapsto(x_1,x_2^2+x_3^2\pm x_4^2).\end{equation}When the sign above is positive, the critical point is known as a \emph{definite fold}; when it is negative, it is known as an \emph{indefinite fold} or \emph{broken singularity}. When there is no ambiguity, we refer to the definite fold locus of a map $h:X\rightarrow Y$ as $S_+(h)$ (or simply $S_+$) and the indefinite fold locus as $S_-(h)$ (or simply $S_-$). For a smooth surface $N^2$, any smooth map $M\rightarrow N$ whose critical locus is a union of Lefschetz singularities and indefinite folds is called a \emph{broken Lefschetz fibration}, and though the results of this paper apply to the case of positive genus, usually $N=S^2$ in the literature concerning broken fibrations. From the local models, it is evident that the fold locus is a smoothly embedded 1-submanifold of $M$ whose image is an immersed 1-submanifold of $S^2$; see also \cite{B1} concerning the topology of indefinite folds. It will be convenient to view a deformation $f_t:M^4\rightarrow S^2$ as a map $F:M_I\rightarrow S^2_I$ such that the image of each slice $F|{M_t}$ is contained in $S^2_t$. For such maps (indeed for general maps from 5-manifolds to 3-manifolds), it turns out that folds are obtained by taking the direct product of the previous model with $\R$. That is, up to change of local coordinates, fold points of deformations are given by the following equation:\begin{equation}\label{eq:fold2}(t,x_1,x_2,x_3,x_4)\mapsto(t,x_1,x_2^2+x_3^2\pm x_4^2)\end{equation} with the same dichotomy between definite and indefinite folds \cite{AGV}.
\begin{figure}\centering\begin{minipage}[t]{0.5\linewidth}\centering\includegraphics{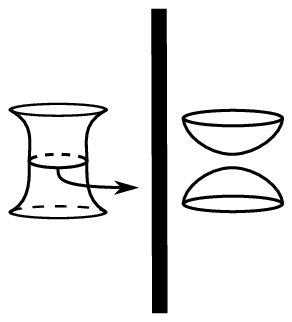}\caption{Base diagram of an indefinite fold} \label{fold}\end{minipage}\begin{minipage}[t]{0.5\linewidth}\centering\includegraphics{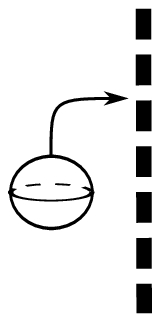}\caption{Base diagram of a definite fold.} \label{deffold}\end{minipage}\end{figure} For a fixed value of $t$, the fibers near an indefinite fold are those of an indefinite morse singularity, or in other words a morse singularity whose index is neither maximal nor minimal. The map naturally parametrizes the target disk via the coordinates $(x_1,x_2^2+x_3^2-x_4^2)$, and for decreasing values of the coordinate function $x_2^2+x_3^2-x_4^2$, such a singularity corresponds to a 3-dimensional 2-handle (for increasing values, a 3-dimensional 1-handle). This is illustrated by examining the point preimages of a horizontal arc parametrized to go from left to right in Figure \ref{fold}. As one travels from left to right, a circle depicted on the fiber at the left shrinks to a point, resulting in a nodal fiber lying over each point in the vertical arc. Then the fiber separates into two disks as shown on the right. This is a convenient notation for base diagrams first employed by Auroux in which the arrow gives the orientation of the indefinite arc as well as specifying which circle contracts as one approaches the fold. As in Example 2 of \cite{ADK}, we refer to this circle as the \emph{vanishing cycle} associated to the indefinite fold. One might use the more precise \emph{round vanishing cycle} when it is important to distinguish between those of indefinite folds and Lefschetz critical points. Figure \ref{deffold} depicts the image of a definite fold. Again, the arrow points in the direction of decreasing values of $(x_2^2+x_3^2+x_4^2)$; indeed the fiber is empty where this coordinate would be negative. The preimage of each point in the vertical arc is a point, which expands into a sphere moving toward the left. For those who are familiar with the interplay of near-symplectic structures with broken fibrations, it is immediately clear that definite folds are not compatible with near-symplectic structures, as this sphere is nullhomologous. In some sense, the spherical fibers in the picture are themselves vanishing cycles. With this in mind, it is natural to draw an arrow from the sphere to the image of the definite fold at which that sphere contracts. This specificity can be rather useful when dealing with base diagrams involving multiple fold arcs and fiber components.\end{subsubsection}
\begin{subsubsection}{Cusps}The second-highest stratum of the critical manifold of a deformation is the (1-dimensional) closure of the cusp locus, denoted $S^{(3,1)}$, which can be singled out as the critical manifold of the map obtained by restricting a deformation to its own critical locus. A cusp point of a map $M^4\rightarrow S^2$ has the local model $$(x_1,x_2,x_3,x_4)\mapsto(x_1,x_2^3-3x_1x_2+x_3^2\pm x_4^2).$$ When the sign above is negative, the cusp is adjacent to two indefinite fold arcs as in Figure \ref{cusp}; when positive, one arc is definite and the other indefinite as in Figure \ref{defcusp}. In each of Figures \ref{cusp} and \ref{defcusp} the critical locus is a smoothly embedded curve in $D^4$ consisting of two arcs of fold points which meet at an isolated cusp point. Similar to folds, the cusp locus of a deformation has the local model obtained by crossing with $\R$: $$(t,x_1,x_2,x_3,x_4)\mapsto(t,x_1,x_2^3-3x_1x_2+x_3^2\pm x_4^2).$$Another way to understand the local model for a cusp is by considering the family of restrictions of the maps in Figures \ref{cusp} and \ref{defcusp} to the preimages of vertical arcs, which describes a homotopy $D^3_{[-\epsilon,\epsilon]}\rightarrow[-\epsilon,\epsilon]_{[-\epsilon,\epsilon]}$ that traces the formation of a pair of canceling morse critical points when the family is seen moving from left to right. For this reason, the indices of the critical points must differ by 1, which implies any cusp involving a definite fold (index 0 or 3) as one of its constituent arcs must also involve a fold whose corresponding morse index is 1 or 2 (an indefinite fold). Thus there is only one kind of cusp involving definite folds up to local parametrization, the one shown in Figure \ref{defcusp}. As there is no ambiguity, we call this a \emph{definite cusp}.\begin{figure}\centering\begin{minipage}[ht]{0.5\linewidth}\centering\includegraphics{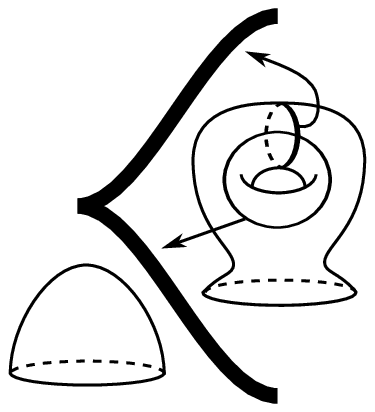}\caption{The indefinite cusp.} \label{cusp}\end{minipage}\begin{minipage}[ht]{0.5\linewidth}\centering\includegraphics{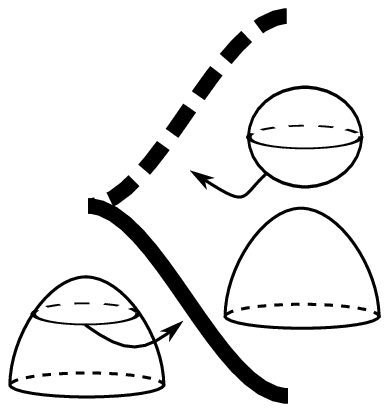}\caption{The definite cusp.} \label{defcusp}\end{minipage}\end{figure} The fibers above key parts of the target disk are shown in the figures: in both figures the fiber above the actual cusp point is homeomorphic to a disk. In Figure \ref{cusp}, the vanishing cycles are a pair of simple closed curves in a punctured torus that transversely intersect at a unique point. See also \cite{L,Lev} concerning cusp points.\end{subsubsection}
\begin{subsubsection}{Swallowtails} The lowest stratum for deformations does not appear in a stable map $M^4\rightarrow S^2$ for dimensional reasons. The locus of swallowtails has Boardman symbol $(3,1,1)$ and consists of a finite collection of points, each of which has the local model \begin{equation}\label{eq:swallowtail}(t,x_1,x_2,x_3,x_4)\mapsto(t,x_2^4+x_2^2t+x_1x_2\pm x_3^2\pm x_4^2).\end{equation}Up to right-left equivalence, there are three types of swallowtails, distinguished for the purposes of this paper by the fact that each produces a pair of cusps and thus three arcs of fold points in a base picture. Local behavior for the three swallowtails is discussed below.\end{subsubsection}

Notably, Lefschetz critical points are not included in the preceding list. This is because they are unstable (see, for example, Figure 6 of \cite{L}). As remarked above, the classification of stable critical loci (both for general stable maps \cite{AGV} and for $k$-parameter families of maps for small $k$ \cite{W}) has been completed in low dimensions. For this reason, one may define the stability of a map by restricting the form its critical locus may take. By the results of \cite{W}, for deformations the classification consists of the list above: folds, cusps, and swallowtails. Thus a smooth homotopy $M_I\rightarrow S^2_I$ is a deformation if and only if each point in its critical locus has a local model chosen from one of these. A deformation $f$ is itself a stable map, but this does not imply that each slice $f|_{M_t}$ is stable. For a deformation, at a finite collection of points in $M_I$ there are critical points which are unstable when considered in the context of maps from smooth 4-manifolds to surfaces, but which become stable when exhibited in 1-parameter families. What follows is a list of these possibilities (which is complete by Theorem 4.4 of \cite{L}), exhibited as moves one can perform on a base diagram. Each move is given locally as a homotopy of a stable map $D^4\rightarrow D^2$, supported away from the boundary spheres of the source and target spaces.\end{subsection}
\begin{subsection}{Catalog of moves involving indefinite folds}\label{icat}The following modifications appear in \cite{L}; they appear below for convenience and a complete exposition.
\begin{subsubsection}{Isotopy}The first member of the list is perhaps the most obvious: called \emph{isotopies} in Theorem 4.1 of \cite{L}, one may perform a $(1,1)$-stable homotopy of a map $M\rightarrow S^2$ in which the stratified isotopy class of its critical locus is unchanged. In a base diagram, one sees the critical set moving around after the fashion of an isotopy of knot diagrams. In practice, showing that a given modification of base diagrams is of this type can be subtle; however isotopies offer a surprising variety of possible modifications of a given map. Stated precisely, there is a notion of \emph{local left-right equivalence} where maps $f,g:X\rightarrow Y$ are equivalent if there is an open cover $\{U_i\}$ of $X$ such that each restriction $f|_{U_i}$ is right-left equivalent to $g|_{U_i}$. Then an isotopy of $f$ is a one-parameter family of locally right-left equivalent maps.\end{subsubsection}
\begin{subsubsection}{Birth}The local model for births appears in \cite{EM} as follows:\begin{equation}\label{eq:birtheq}(t,x_1,x_2,x_3,x_4)\rightarrow(t,x_1,x_2^3+3(x_1^2-t)x_2+x_3^2-x_4^2).\end{equation} For $t<0$ the map has empty critical locus. For $t\geq0$, the critical locus is $\{x_1^2+x_2^2=t,x_3=x_4=0\}$. For $t=0$ the local model is similar to that of a cusp, and it will be important to know that it actually is a cusp, so the fact appears in the following proposition. In this paper (especially in the proof of Lemma \ref{simplemma}), this cusp is called an indefinite (or, as discussed below, a definite) \emph{birth point}.
	\begin{prop}\label{birthcusp}For the birth deformation, the origin $\R\times\R^4$ is a cusp point.\end{prop}
	\begin{proof}Temporarily denote the map of Equation \ref{eq:birtheq} by $F$ and its critical locus given above by $S$. Because $F$ is a deformation, Proposition \ref{defsarestable} implies it is a stable map $\R^5\rightarrow\R^3$ and it is known that the condition to verify is that the point in question lies in $S^{(3,1,0)}(F)$ (see, e.g., Section 3 of \cite{An2}). The restriction $F|_S$ can be written as a map $F|_S:\R^2\rightarrow\R^2$ given by $(x_1,x_2)\mapsto(x_1,-2x_2^3)$. At the origin the derivative of this map has rank 1, putting that point in $S^{(3,1)}$. Further restricting $F|_S$ to its own critical locus $\{x_2=0\}$ yields the identity map on $x_1$ which has empty critical locus. Thus the origin is in $S^{(3,1,0)}(F)$ as desired.\end{proof}It is known (\cite{EM,L}) that for the slice $\{t=\epsilon>0\}$ the critical locus is a circle composed of two open arcs of indefinite fold points connected to each other by two cusp points as in Figure \ref{birth}. Thus the critical locus of this move in $D^4_{[-\epsilon,\epsilon]}$ can be described as a hemisphere with a cusp equator. More precisely, taking the projection $T:S\rightarrow[0,1]$ to the $t$-axis as a morse function, Theorem 4.4 of \cite{L} implies that any index 0 (or index 2 replacing $t$ with $-t$) critical point of this function must correspond to a birth, which either must be of this type, or a ``definite birth" described below. \begin{figure}[ht]\begin{center}\includegraphics[width=0.4\linewidth]{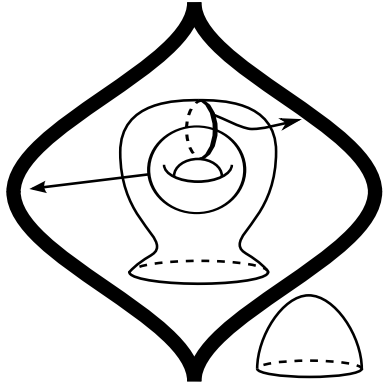}\end{center}\caption{A birth model that creates a circle of critical points.}\label{birth}\end{figure}The fibration structure for the birth outside of the slice $\{t=0\}$ is as follows: for $t<0$ it is a trivial disk bundle. For $t>0$, the fibers outside the circle are disks, while the fibers inside the circle are punctured tori with vanishing cycles as shown in Figure \ref{birth} when traveling to the left and right from a reference point in the center of the circle.\end{subsubsection}
\begin{subsubsection}{Merge}In the same way that a birth corresponds to an index 0 or 2 critical point of the projection of $S$ to the $t$-axis, the merging move corresponds to an index 1 critical point. The local model in the indefinite case is as follows:\begin{equation}\label{eq:mergeeq}(t,x_1,x_2,x_3,x_4)\mapsto(t,x_1,x_2^3+3(t-x_1^2)x_2+x_3^2-x_4^2).\end{equation} The critical locus is given by $\{x_2^2-x_1^2=t,\ x_3=x_4=0\}$ which parametrizes a saddle. Similar to the birth move, there are two obvious cusp arcs which meet at the origin at a point which turns out to also be a cusp, and in this paper (especially in the proof of Lemma \ref{simplemma}), this point is called an indefinite (or, as discussed below, a definite) \emph{merge point}.
	\begin{prop}\label{mergecusp}For the merge deformation, the origin in $\R\times\R^4$ is also a cusp point.\end{prop}
	\begin{proof}Temporarily denote the map of Equation \ref{eq:mergeeq} by $F$ and its critical locus given above by $S$. Because $F$ is a deformation, Proposition \ref{defsarestable} implies it is a stable map $\R^5\rightarrow\R^3$ and like Proposition \ref{birthcusp} it is sufficient to show that the point in question lies in $S^{(3,1,0)}(F)$. The restriction $F|_S$ can be written as a map $F|_S:\R^2\rightarrow\R^2$ given by $(t,x)\mapsto(t,-2x^3)$. At the origin the derivative of this map has rank 1, putting that point in $S^{(3,1)}$. Further restricting $F|_S$ to its own critical locus $\{x=0\}$ yields the identity map on $t$ which has empty critical locus. Thus the origin is in $S^{(3,1,0)}(F)$ as desired.\end{proof}The proposition shows that the critical locus of a merge is a disk composed of fold points which is bisected by an arc of cusp points. Thus, as stratified surfaces the critical loci of the merge and birth deformations are diffeomorphic. In this paper, a merging move is always depicted in a base diagram along with a curve along which the move occurs. When the curve is between two fold points, this signals that a merge is possible between those points and occurs in the interval before the next picture; when it goes between a pair of cusps, it signals an inverse merge in the same way.\begin{figure}[ht]\begin{center}\includegraphics[width=\linewidth]{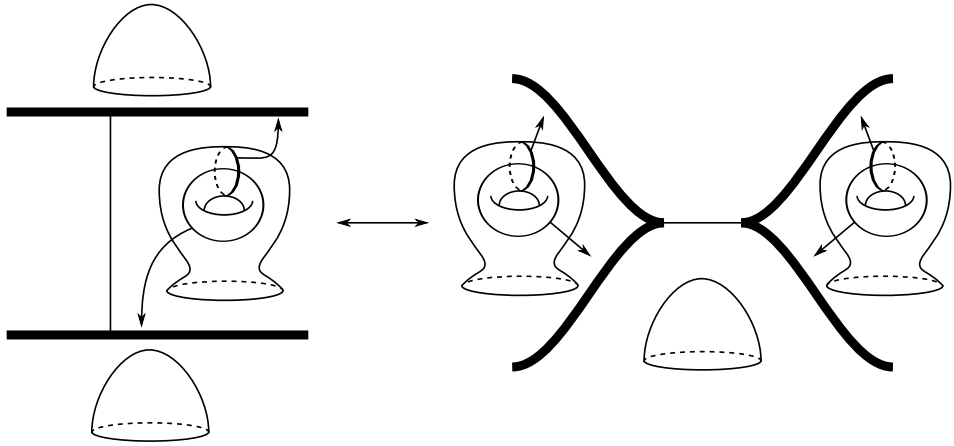}\end{center}\caption{A merging move between two indefinite folds.}\label{merge}\end{figure} The fibration structure for the merge outside of the slice $\{t=0\}$ is as follows: for $t<0$ the base diagram appears as in the left side of Figure \ref{merge}, where the fibers above the central region are punctured tori, and the fibers above the top and bottom regions are disks. Similar to the birth deformation, the fibration structure above an arc that connects the top and bottom regions traces a morse function on the 3-disk with canceling morse critical points of index 1 and 2. In other words, for a merge to occur, the vanishing cycles as obtained by following the vertical arc drawn on the left of the figure in either direction must intersect at a unique point. The fibration structure for $t>0$ includes two cusp points, with the vanishing cycles as shown. Arguably finding more application in the literature is the reverse of this move, which requires a \emph{joining curve} connecting the two cusps (whose image appears as the horizontal arc in the right side of the figure) that intersects the critical locus precisely at the two cusps it joins (see, e.g., 4.4 of \cite{Lev}). Generalizations of the reverse of this deformation to larger and smaller dimensions in the source and target spaces have been used at various times in the past, appearing for example in \cite{Lev}, where Figure 3 has a rather illuminating picture of what is actually happening with the merging move, and others, e.g. \cite{S}.\end{subsubsection}
\begin{subsubsection}{Flip}Here appears a type of critical point involving more than cusp and fold points. Figure \ref{flip} depicts the formation of two cusps in the critical locus, which now has a loop in its image. Unlike the previous two deformations, these cusp arcs meet not at a cusp point but a higher-order critical point, an indefinite swallowtail.
\begin{figure}[ht]\begin{center}\includegraphics{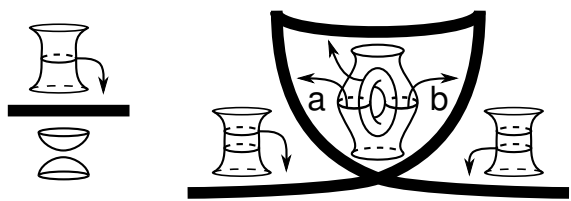}\end{center}\caption{The indefinite flipping move, also known as the result of passing though an indefinite swallowtail.}\label{flip}\end{figure}The deformation begins with an arc of indefinite fold points. After passing the swallowtail point, the fiber above a point within the loop is a twice-punctured torus, with vanishing cycles as shown. The preimage of the point where the two fold arcs intersect is also a twice-punctured torus, except with two nodal singularities obtained by shrinking the two disjoint vanishing cycles labeled $a$ and $b$.\end{subsubsection}
\begin{subsubsection}{Wrinkle}This modification is not prominent in the singularity literature because it involves critical points which are not stable under small perturbation (as shown by the existence of the wrinkling move itself). Thus without loss of generality it is not present in a deformation, though it is an explicit way to modify any broken Lefschetz fibration into a stable map whose critical locus consists of cusps and indefinite folds. It appears in this paper because it is used in Section \ref{apps}. Beginning with the local model for a Lefschetz critical point given by $(z_1,z_2)\mapsto z_1z_2$, the wrinkling homotopy is given by $(z_1,z_2)\mapsto z_1z_2+t\Rp(z_1)$, $t\geq0$. \begin{figure}[ht]\begin{center}\includegraphics{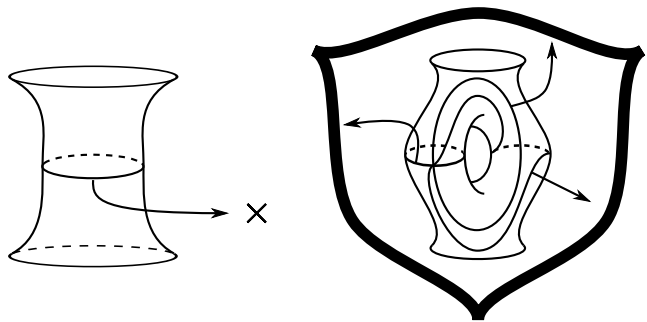}\end{center}\caption{Lekili's wrinkling move, given by an infinitesimal deformation of a Lefschetz critical point.}\label{wrinkle}\end{figure}The picture is largely self-explanatory: beginning with a Lefschetz critical point (which appears at the left as X), a 3-cusped circle of indefinite folds opens up, bounding a region whose fiber has genus higher by 1. It is interesting to note that, traveling in a loop in the annulus-fibered region parallel to the circle of critical points, the Lefschetz monodromy associated to the isolated critical point remains after the perturbation.
\end{subsubsection}
\begin{subsubsection}{Sink}Putting these moves together, it is possible to convert any stable map whose critical locus consists of cusps and indefinite folds into a broken Lefschetz fibration by a $C^0$-small homotopy; a modification which appears in Figure of \cite{L} gives the prescription. For want of a better term, this paper refers to this homotopy as \emph{sinking} a Lefschetz critical point into the fold locus, and to the reverse of this homotopy as \emph{unsinking} a cusp. The vanishing cycles associated with unsinking a cusp have a simple description: suppose the vanishing cycles near a cusp are simple closed curves $a$ and $b$ as in Figure \ref{cusp}, equipped with some (arbitrary) orientation for each. Then the Lefschetz critical point that comes from unsinking is precisely the one whose monodromy sends $a$ to $b$: its vanishing cycle is homotopic to the concatenation $b+a$.\end{subsubsection}
 Those familiar with the subject will notice that most of these moves are slightly different from those appearing in \cite{L} in that they do not involve the formation or disappearance of Lefschetz critical points. As sinking and unsinking are arbitrarily small perturbations, the difference is small enough for the abuse of terminology to be tolerated, assuming one specifies when sinking and unsinking take place in a given homotopy. As Lefschetz critical points are not stable, yet must persist for some open interval in  $t$ if they appear via unsinking, they do not appear in the course of a deformation. Thus a flip, merge or birth, when appearing in a deformation, is implicitly one of the above moves. With this understood, the main theorem has a precise statement:\begin{thm}\label{T}If two broken Lefschetz fibrations are homotopic, then there exists a homotopy between them which is realized as a sequence of modifications (and their inverses) chosen from the following list: birth, merge, flip, sink, wrinkle and isotopy.\end{thm}\end{subsection}
\begin{subsection}{Catalog of moves involving definite folds}\label{dcat}With some understanding of the moves and local models above, it is straightforward to deduce their definite counterparts.
\begin{subsubsection}{Definite birth}Given locally by the model\begin{equation}\label{eq:defbirtheq}(t,x_1,x_2,x_3,x_4)\rightarrow(t,x_1,x_2^3+3(x_1^2-t)x_2+x_3^2+x_4^2),\end{equation}the modification in Figure\begin{figure}\begin{center}\includegraphics{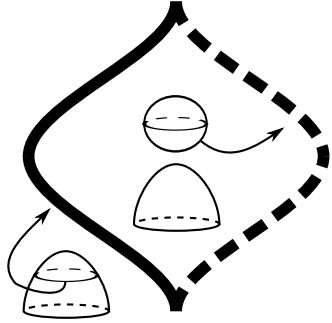}\end{center}\caption{The definite birth move. (defbirth)}\label{defbirth}\end{figure} \ref{defbirth} adds a nullhomologous sphere component to the fiber for points on the interior of the circle. Outside the circle, the fibers are disks. Traveling from left to right along the middle of the figure, one encounters a nullhomotopic vanishing cycle which pinches off a sphere component. Continuing to the right, that same sphere shrinks to a point and disappears. By an argument which is almost verbatim that of Proposition \ref{birthcusp}, the critical locus in $D^4_{[-\epsilon,\epsilon]}$ is a disk bisected by cusps; the only difference is that one side of this disk is swept out by definite folds, the other by indefinite folds.\end{subsubsection}
\begin{subsubsection}{Definite merge}The deformation corresponding to the merging of an arc of definite folds into an arc of indefinite folds (Figure \ref{defmerge}) \begin{figure}[ht]\begin{center}\includegraphics[width=\linewidth]{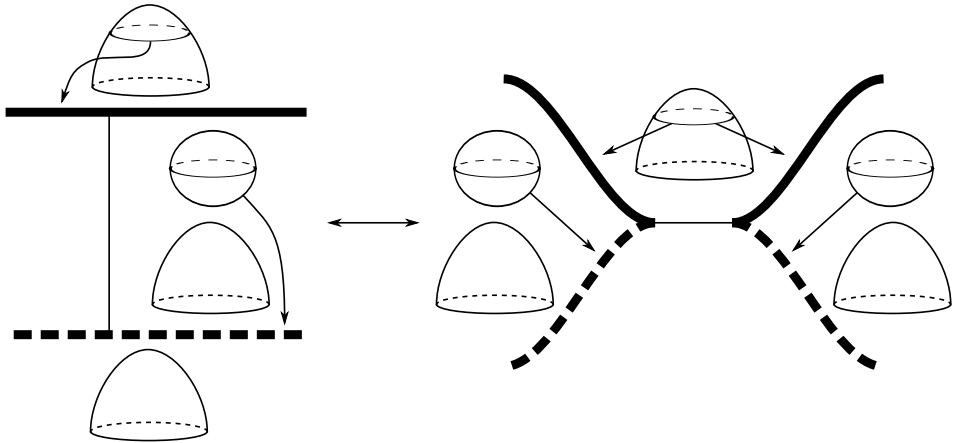}\end{center}\caption{A merging move involving a definite fold.}\label{defmerge}\end{figure} has following local model:\begin{equation}\label{eq:defmergeeq}(t,x_1,x_2,x_3,x_4)\mapsto(t,x_1,x_2^3+3(t-x_1^2)x_2+x_3^2+x_4^2).\end{equation}Like its indefinite counterpart, it corresponds to an index 1 morse critical point, taking $t$ as a morse function $S\rightarrow[0,1]$. It appears as a saddle $S=\{x_1^2=x2^2=t, x_3=x_4=0\}$ in $D^4_{[-\epsilon,\epsilon]}$, and by an argument analogous to that of Proposition \ref{mergecusp}, a definite cusp arc along the slice $\{x_2=0\}\subset S$. Finally, an arc between a definite fold point and an indefinite fold point, or between two definite cusps, signals a definite merge in the same way as in the indefinite case. Note that the vanishing cycles for the two indefinite arcs on the right side are the same: this reflects the fact that they must be nullhomotopic curves in the same fiber component. Lastly, a joining curve for inverse merge must also satisfy the obvious compatibility requirement where the indefinite arcs patch together and the definite arcs patch together in the resulting base diagram.\end{subsubsection}
\begin{subsubsection}{Definite swallowtails}There are two swallowtails that involve definite folds in their local models. The first occurs as in Figure \ref{idi} beginning with an arc of indefinite fold points. The loop that forms is on the opposite side of the arc than the loop that forms by a flip, and the parallel vanishing cycles introduce a nullhomologous sphere component in the fiber. Traveling along the loop of critical points that results from passing through a swallowtail of this type, one can read off the type of fold as indefinite, definite, then indefinite. For this reason these swallowtails are called IDI definite swallowtails.
\begin{figure}[ht]\centering\begin{minipage}[t]{0.5\linewidth}\centering\includegraphics{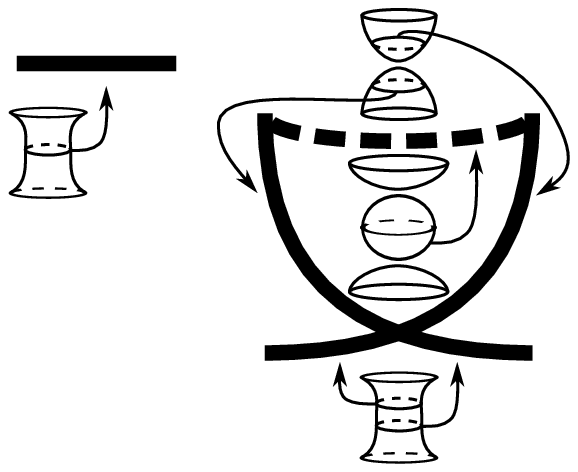}\caption{The definite swallowtail as it appears on an indefinite fold.} \label{idi}\end{minipage}\begin{minipage}[t]{0.5\linewidth}\centering\includegraphics{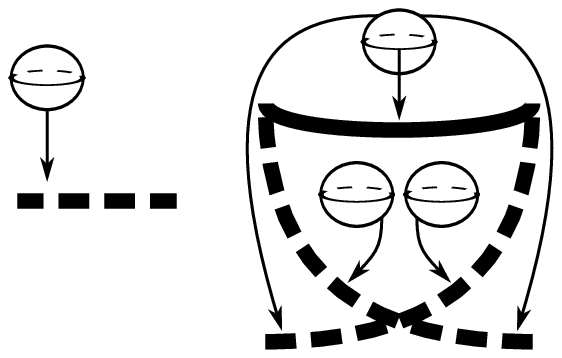}\caption{The definite swallowtail as it appears on a definite fold.} \label{did}\end{minipage}\end{figure} The other definite swallowtail appears in Figure \ref{did}, where the fiber is empty at the bottom of both sides of the diagram. It begins with an arc of definite fold points and opens upward into the region fibered by spheres, introducing another sphere component in the fiber. Each sphere component is necessarily nullhomologous as it dies at some definite fold arc. Following the pattern above, these swallowtail points will be called DID definite swallowtails. In this paper, these two critical points are consistently referred to as definite swallowtails, while the term ``flip" is reserved for the indefinite version appearing in Section \ref{icat}.\end{subsubsection}\end{subsection}
\begin{subsection}{The critical manifold of a deformation}\label{criticalmanifold}Collecting these ideas, the critical locus of a deformation $f_t:M_I\rightarrow S^2_I$ is a smooth surface $S$ properly embedded into $M_I$ with a stratification into $0$, $1$, and $2$-dimensional submanifolds, each contained in the boundary of the next, and the critical locus of the restriction of a deformation to the closure of any stratum is the union of the lower-dimensional strata.
	\begin{df}A stable map from a smooth 4-manifold to a surface whose critical locus consists of indefinite folds and cusps is called a \emph{wrinkled fibration}. When a deformation has empty definite fold locus, it is called a \emph{deformation of wrinkled fibrations}.\end{df}
	The stability of a deformation implies that one may take the projection $T:\crit(f)\rightarrow I$ as a morse function whose index 0 and 2 critical points correspond to definite and indefinite births, while the index 1 critical points correspond to definite and indefinite merging moves. Other than these, the only points at which a deformation passes through an unstable map are the swallowtail points. Since these moves and isotopies are obviously all deformations and any deformation has a corresponding family of base diagrams, the converse statement that a map whose critical manifold obeys these constraints is necessarily a deformation also holds; that is, since its critical locus is of the appropriate form, any 1-parameter family of base diagrams given by a sequence of isotopies and the above moves must correspond to some deformation.\\

A convenient depiction of the critical locus $S$ of a deformation $f$ is as a surface colored according to critical point type. The restriction $g=f|_S$ is locally a continuous embedding (in fact an immersion away from $\crit(g)$, which in our case consists of the cusp and swallowtail points), so a further piece of data is the set of points in $D\subset S$ on which $g$ fails to be injective. Like any immersion of a curve into a surface, the critical images $f_t|_S\subset S^2$ evolve continuously with $t$ in the manner prescribed by the Reidemeister moves, where $f^{-1}(D)$ appears as a union of copies of the diagrams in Figure \begin{figure}[ht]\begin{center}\includegraphics[width=\linewidth]{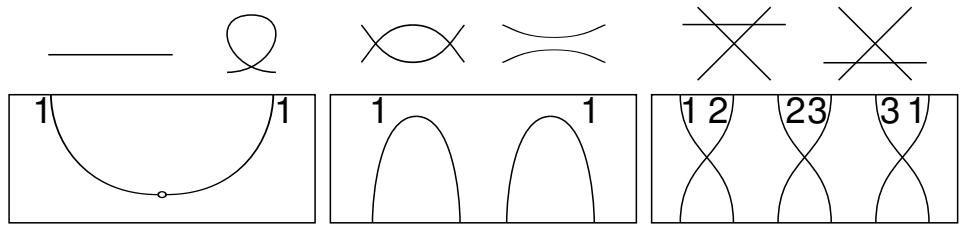}\end{center}\caption{The preimages of intersection points for Reidemeister moves 1, 2 and 3 in the critical locus of a deformation, where the parameter $t$ points either upward or downward.}\label{reid}\end{figure}\ref{reid} and their reflections in $t$ (in the figure, two arcs with the same number have a common image under the deformation and may appear in different components of $S$). More precisely, looking at Figure \ref{reid} as a collection of rigid tangles, $D$ appears as a concatenation in which the components of two tangles may be joined by arcs which project diffeomorphically to the $t$-axis, and such that the numbering is consistent. The following lemma will be important in Section \ref{simp}:
\begin{lemma}\label{r1}For a deformation, the double-point locus given by the two half-open arcs appearing in Figure \ref{reid} corresponding to the Reidemeister-1 move necessarily have a common limit point $p\in\crit(f)$ which is a (definite or indefinite) swallowtail point.\end{lemma}
\begin{proof}Restricting the deformation to the critical locus in a neighborhood of $p$, it becomes a map from a surface into $\R^3$. The point $p$ is precisely the point at which the move occurs, and is a branch point of the map, a generic singularity of such maps. In other words, $p$ is a critical point of $f|_{\crit(f)}$ and therefore it must be a cusp or swallowtail point of the deformation, according to the Boardman stratification. Since the local model for cusps does not involve an immersion of the critical locus, it must be a swallowtail.\end{proof}\end{subsection}\end{section}

\begin{section}{Proof of the main theorem}\label{pf}As discussed above, if a deformation has empty definite fold locus, then it is possible to connect its endpoints by a sequence of the moves of Section \ref{icat}. Thus one approach is to begin with an arbitrary homotopy between broken Lefschetz fibrations, perturb it to be $(1,1)$-stable, and remove the definite fold locus while preserving the deformation condition. Note that it is not necessary that the modification itself be realized by a homotopy: all that is required for Theorem \ref{T} is an existence result for a deformation of wrinkled fibrations, regardless of the homotopy class of the deformation map $f$. If these modifications were realizable as a sequence of homotopies of $f$, one could easily obtain an ``elimination of definite fold" result analogous to \cite{S} for maps from 5-manifolds to 3-manifolds, as discussed in Section \ref{apps}. The modification begins by removing the definite swallowtails. Inspired by a paper of Ando \cite{An1} which indicates that swallowtails of all kinds cancel in pairs, the first step is to in some sense replace all definite swallowtails by flips so that the closure of the definite fold locus becomes a surface with nonempty boundary consisting of a union of cusp circles. In Section \ref{simp}, a kind of surgery on this surface renders a deformation such that $S_+$ becomes a union of disks, on each of which the deformation is an embedding which is as simple as possible. The last step is a trick to cause these disks to appear in definite births which are trivial in a way that allows them to be omitted. The following examples will be used in the argument, and are typical of the kind of manipulations that form the heart of this paper.
\begin{ex}\label{birthmerge}The result of passing through a pair of swallowtails can also occur by choosing a particular birth, then merging one of its fold arcs with a preexisting fold arc as in Figure \ref{bmpic}.
\begin{figure}[ht]\begin{center}\includegraphics[width=\linewidth]{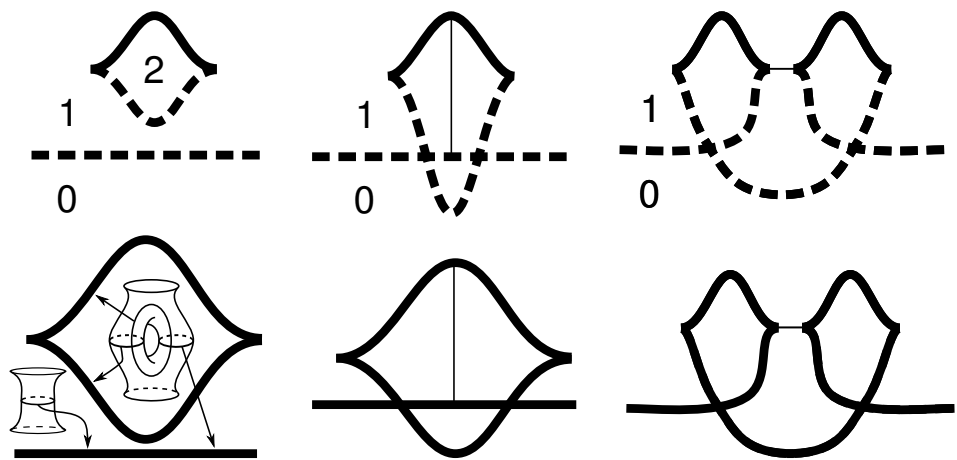}\end{center}\caption{Two swallowtails have the same effect as a birth followed by a merge. There are two definite versions (one of which appears at the top) and an indefinite version (bottom).}\label{bmpic}\end{figure} 
It is helpful to follow the course of the deformation both from left to right and in reverse. From left to right, it begins with the result of a birth which is modified by an isotopy followed by a merge. From right to left, a fold arc appears as it would just after the deformation has passed through two swallowtails. Then there is an inverse merge as indicated, and an isotopy. In some sense, the loops (and thus the swallowtail critical points not shown) at the right might be viewed as defining the location at which the births at the left occur in the larger fibration, even though (assuming the total space is connected) any pair of births are related by isotopy because they are homotopies supported on 4-balls in $M$.\end{ex}
\begin{ex}\label{tradex}This example is due to David Gay, who employed a clever use of the modification that appears in Figures 5 of \cite{B2} and 11 of \cite{L}. Employing Example \ref{birthmerge}, the author has slightly changed its presentation to avoid swallowtails, and it appears in the proof of the main theorem more than once. The effect is to switch between definite and indefinite circles by a deformation $(D^3\times S^1)_I\rightarrow D^2_I$ as shown in Figure \ref{trade}.
\begin{figure}[ht]\begin{center}\includegraphics[width=\linewidth]{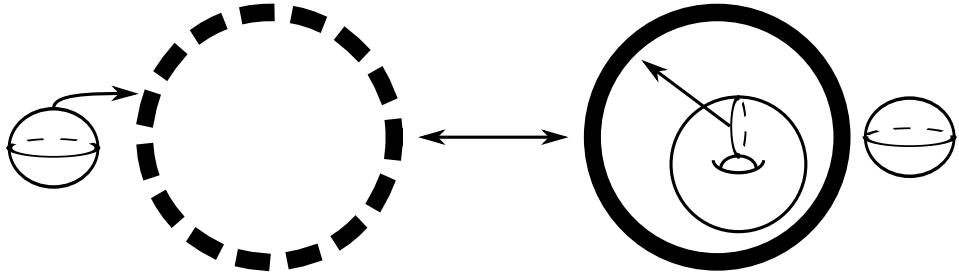}\end{center}\caption{Interchanging definite and indefinite circles.}\label{trade}\end{figure}
	Here the base diagram on the left has integers indicating the number of sphere components of the fiber. Within the circle of definite folds, the point preimages are empty, and outside of the circle the fiber consists of a 2-sphere. At the right, the critical locus consists of a circle of indefinite fold points. Inside the circle the fiber is a torus, and passing outside this circle an essential curve shrinks as the vanishing cycle, resulting in a fibration by spheres. The intermediate steps appear in Figure \ref{tradepf},
\begin{figure}[ht]\begin{center}\includegraphics[width=\linewidth]{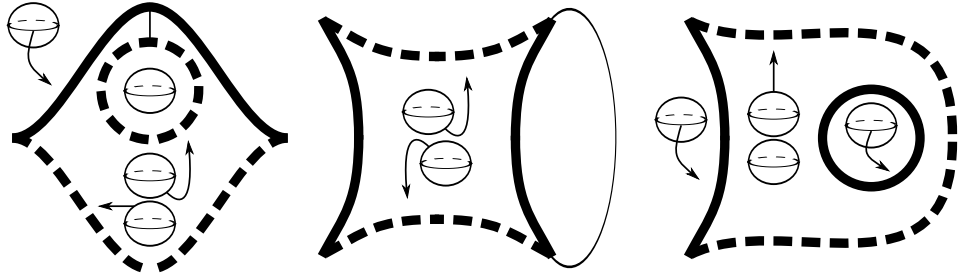}\end{center}\caption{Switching between definite and indefinite circles.}\label{tradepf}\end{figure}
	in which a definite birth as in the previous example has taken place outside the definite circle, followed by an isotopy in which the newly introduced region of sphere fibers expands to result in the left side of the figure. Performing the merge (which is the one from Example \ref{birthmerge}) and then an inverse merge as indicated gives the right side of the figure. Here, the indefinite circle may be pushed past the other indefinite arc in an isotopy such that each spherical fiber above the region it bounds experiences surgery on a pair of points as the other indefinite arc passes by, resulting in torus fibers. Removing the cusped circle by an inverse definite birth finally results in the right side of Figure \ref{trade}.\end{ex}
\begin{rmk}It is interesting to note that there is something subtle going on with these modifications. Closing off the boundary fibration in Example \ref{tradex} with a copy of $S^2\times D^2$ results in a family of fibrations of the 4-sphere, where it is known that the gluing data associated with the fibration uses a nontrivial loop of diffeomorphisms of the torus (in the language of \cite{B1}, the round handle corresponding to the indefinite circle is \emph{odd}), and it would be interesting to clarify how and when this manifests itself in the above deformation.\end{rmk}
\begin{rmk}As the modification of Example \ref{tradex} is entirely local around a definite fold, it seems likely that it could lead to an alternate proof of Theorem 2.6 of \cite{S} in the special case of 4-manifolds, and yet another slightly novel existence result for broken Lefschetz fibrations, but this is beyond the scope of this paper.\end{rmk}
\begin{subsection}{Removing definite swallowtails}\label{swallowtails} The definite swallowtails are precisely those isolated points in the lowest-dimensional stratum which are adjacent to definite folds. It is clearly necessary to remove them, as their local models involve definite folds. The recurring theme in this paper is a kind of surgery on the map itself; that is, the map and the fibration structure it induces are equivalent data, and performing a fibered surgery in which some fibered subset of a manifold is replaced by a diffeomorphic subset with a different fibration structure on its interior, yet an isomorphic fibration structure on its boundary can be viewed as a modification of the map that preserves the manifolds involved. The first instance of this is as follows. 
\begin{lemma}\label{defswallowtails}Suppose $f:M_I\rightarrow F_I$ is a deformation. Then there exists a deformation $f':M_I\rightarrow F_I$ such that $ f_i=f'_i$, $i=0,1$, and such that $S^{(3,1,1)}(f')$ consists entirely of indefinite swallowtails.\end{lemma}
\begin{proof}The IDI definite swallowtail occurs as a single critical point whose local model appears in Figure \ref{idi}. As described above, the tactic is to give a deformation $g_{idi}$ with the same boundary fibration, but such that $S^{(3,1,1)}(g_{idi})$ is free of definite swallowtails. Removing a neighborhood of a definite swallowtail and gluing in this map causes the deformation to go through these steps instead of a definite swallowtail. \begin{subsubsection}{IDI-type swallowtails}For the purposes of exposition, it seems more natural to present the deformation in reverse: the description begins (and $g_{idi}$ ends) with the right side of Figure \ref{idi}. 
\begin{figure}[ht]\begin{center}\includegraphics[width=\linewidth]{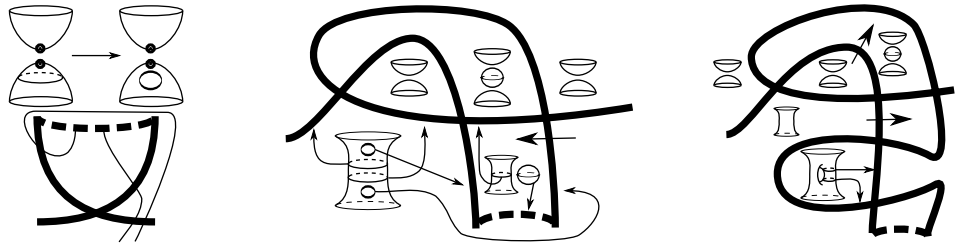}\end{center}\caption{Describing the deformation $g_{idi}$ (part 1).}\label{idi1}\end{figure}To begin, the left side of Figure \ref{idi1} has two arcs that pass through the images of various fold points. The family of point preimages over each of these arcs can be described by the decorated disks at the top left of the picture as follows. One either arc, choose a point in the region fibered by pairs of disks. Following the arc toward the definite fold endpoint, a vanishing cycle which appears as a circle on one of the disks shrinks to a point, pinching off a sphere component from that disk (going along the other arc, the vanishing cycle appears on the other disk), which then shrinks to a point at the definite fold. Going in the other direction, the two disks connect sum at a pair of points to form cylindrical fibers. The first step is an isotopy in which the vanishing cycle moves over to bound a disk in the fiber that does not contain the connect sum points. Each of the arcs can be interpreted as the image of a morse function $D^3\rightarrow I$ and it is a useful exercise to envision how this isotopy is reflected by pushing around a 1-parameter family of balls in $\R^3$. This interpretation involves three copies $B_1,B_2,B_3\cong D^3$ connected by 1-handles $B_1$ to $B_2$ and $B_2$ to $B_3$, modified by a handleslide that slides a foot of the $(B_2,B_3)$ handle over the other 1-handle to lie on $B_1$ (for the other arc, simply switch the indices 1 and 3 in the description of the handleslide). This isotopy is supported away from the intersection point of the two indefinite arcs. With this understood, one may perform another isotopy to obtain the middle of the figure, in which the indefinite arcs adjacent to the definite arc have vanishing cycles that bound disks in the cylindrical fiber. In other words, leaving the region fibered by $S^2\sqcup S^1\times I$ by crossing one of the indefinite arcs on either side results in a connect sum between the sphere and the cylinder. A further isotopy is indicated in the middle picture by a large arrow to obtain the right side of the figure, where the vanishing cycles on the twice-punctured torus fiber come from connect summing the sphere and the cylinder fiber two times. A further isotopy is indicated on the right side of Figure \ref{idi1} by large arrows, which can be interpreted as an R3 move followed by an R2 move.
\begin{figure}[ht]\begin{center}\includegraphics[width=\linewidth]{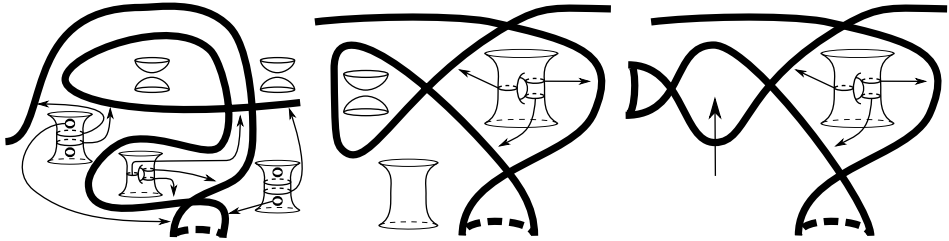}\end{center}\caption{Describing the deformation $g_{idi}$ (part 2).}\label{idi2}\end{figure} Having performed the isotopy, most of the vanishing cycles appearing on the left side of Figure \ref{idi2} are easy to deduce from those of the previous picture. The new vanishing cycle appearing on the twice-punctured torus fiber can be explained by the fact that it is obtained from the adjacent cylindrical fiber by self-connect sum, and the vanishing cycle going around the cylinder survives into the higher-genus fiber by continuity. The middle of Figure \ref{idi2} is just a slightly simpler version of the previous one, obtained by straightening out some of the fold arcs by an ambient isotopy of the base diagram, and a flip results in the right side, where a large arrow indicates an isotopy, an R2 move that removes two intersection points, resulting in the left side of Figure \ref{idi3}. 
\begin{figure}[ht]\begin{center}\includegraphics[width=\linewidth]{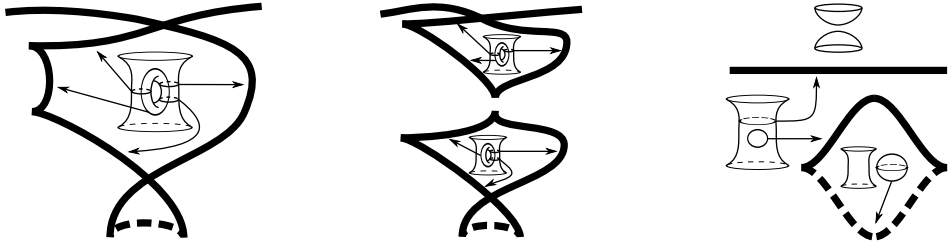}\end{center}\caption{Describing the deformation $g_{idi}$ (part 3).}\label{idi3}\end{figure}The vanishing cycles here may be deduced from the previous image by the fact that the twice-punctured torus fiber is obtained from the $D^2\sqcup D^2$ fiber by connect summing two times as indicated by the vanishing cycles appearing in the right side of Figure \ref{idi2}, where the isotopy is seen sweeping the two indefinite arcs over each point in the $D^2\sqcup D^2$ region. In the left side of Figure \ref{idi3}, the vanishing cycles at the left and right of the twice-punctured torus region intersect transversely at a unique point, and merging along an arc that connects these indefinite folds results in the middle of the figure, where the cusped loops may be removed by inverse flips. The reader will notice that the vanishing cycles within the two loops are different, so there is a question of whether both loops can be removed by inverse flips. Given an indefinite fold arc (call it the ``parent arc" for now), performing a flip changes the fiber on the side with lower Euler characteristic by replacing a fiberwise cylinder neighborhood of the vanishing cycle with the decorated genus-1 surface pictured in Figure \ref{flip}. With this in mind, the two parent arcs in the middle of Figure \ref{idi3} have different vanishing cycles: one is a meridian of the cylinder while the other bounds a disk, leading to their differing appearance. Reversing these flips gives the right side of Figure \ref{idi3}, where the fold circle is easily removed by inverse definite birth. This deformation, in reverse, is called $g_{idi}$, and may be glued in or substituted for a definite IDI-type swallowtail.
Thus the IDI-type definite swallowtails may be removed from any deformation.\end{subsubsection}
\begin{subsubsection}{DID-type swallowtails}For this type of critical point, the map $g_{did}$ does not need to be presented in reverse, and is most naturally presented beginning with the left side of Figure \ref{did}, where a definite birth results in the left side of Figure \ref{did1}. This deformation is easier to describe because the fibers are all spheres and IDI-type swallowtails are freely available for use. 
\begin{figure}[ht]\begin{center}\includegraphics[width=\linewidth]{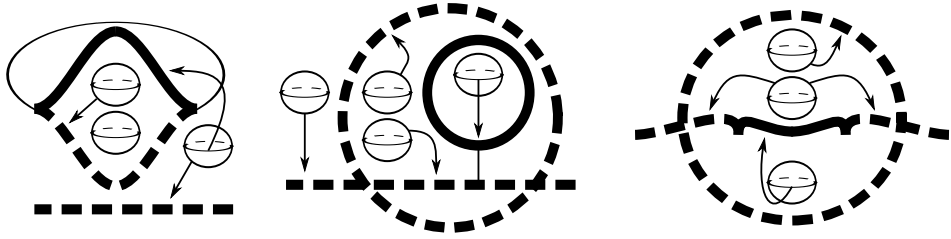}\end{center}\caption{The beginning of the deformation $g_{did}$.}\label{did1}\end{figure}
Performing an inverse merge as indicated followed by an isotopy gives the second diagram of Figure \ref{did1}, and note the merge indicated there, which results in the right side of the figure. Another way to get this same picture is to use Example \ref{birthmerge} followed by inverse merging the two outside cusps in the right of Figure \ref{bmpic}. The left side of Figure \ref{did2} results from an application of the previous argument, performing the deformation of $g_{idi}$ in a neighborhood of an indefinite fold point, with the same effect as passing though an IDI-type swallowtail. 
\begin{figure}[ht]\begin{center}\includegraphics[width=\linewidth]{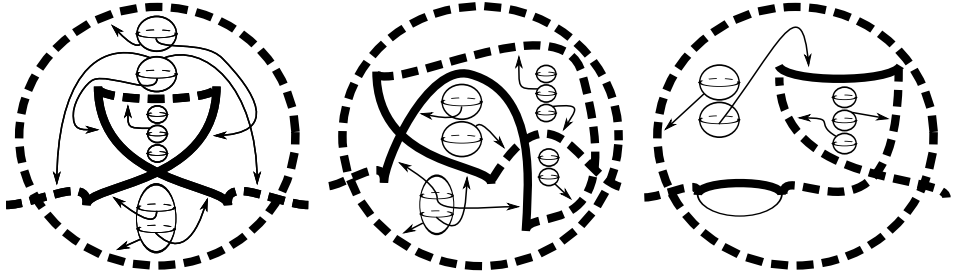}\end{center}\caption{Part of the deformation $g_{did}$.}\label{did2}\end{figure}
	The remaining images in the figure come from isotopies. In the first, a definite cusp passes across a definite fold. In these images, the combinatorial data of which sphere vanishes or appears at each fold arc (that is, the definite vanishing cycle data) is precisely the information that explains the validity of each modification. For each base diagram, in the region fibered by three spheres, the ``middle" sphere dies at the definite fold arc, so that the ``top" and ``bottom" spheres in that region persist as the ``top" and ``bottom" spheres in the region above. With this understood, the isotopy resulting in the middle of Figure \ref{did2} can be interpreted as a retraction of the region fibered by bottom spheres. Another isotopy results in the right side of the figure. In this move, the $\sqcup_2S^2$ fibers above the central region in the middle of Figure \ref{did2} are modified as the top arc passes downward over it by pinching a loop in the top sphere. The vanishing cycle data for the $\sqcup_3S^2$-fibered region of right side is obtainable from that of the previous diagram, where an arc connecting two definite cusps indicates an inverse merge that results in the left side of Figure \ref{did3}.
\begin{figure}[ht]\begin{center}\includegraphics[width=\linewidth]{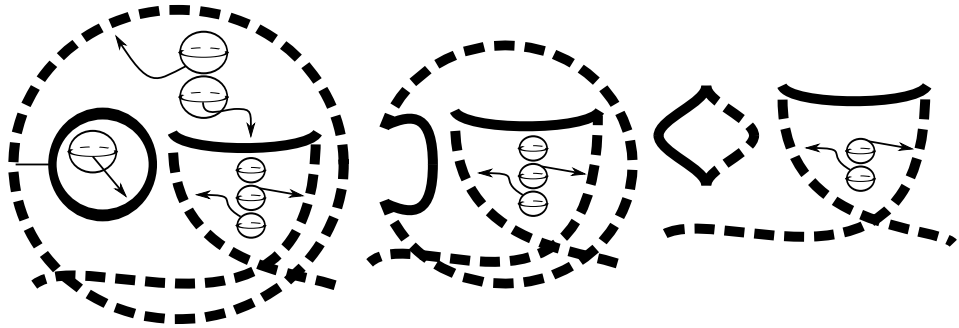}\end{center}\caption{Part of the deformation $g_{did}$.}\label{did3}\end{figure}
	Here, an arc indicates a definite merging move (that in some sense reverses the very first inverse merge that was indicated in the left side of Figure \ref{did1}), which results in the middle diagram. An isotopy results in the right side of Figure \ref{did3}, in which the 4-ball region of the fibration fibered by the ``top" spheres is retracted. An inverse definite birth finally results in the right side of Figure \ref{did}, completing the description of $g_{did}$. As with $g_{idi}$, this deformation is may be employed in place of a DID-type swallowtail. Then $f'$ is obtained from $f$ by appropriately substituting $g_{idi}$ and $g_{did}$ in place of all definite swallowtails.\end{subsubsection}
\end{proof}
\end{subsection}

\begin{subsection}{Simplification of the definite locus}\label{simp}Having eliminated the swallowtails adjacent to definite folds, $S_+$ is now a smooth surface embedded in $M_{(0,1)}$ whose boundary consists of a union of circles, which is precisely the definite cusp locus. Each component of this surface arises in the course of a deformation whose ends are free of definite folds, and the only way for a definite fold to arise in such a deformation is through a definite birth as in Figure \ref{defbirth}. For this reason, the closure of each component of this surface has nonempty boundary consisting of a collection of cusp circles in $M_{(0,1)}$. Now is a good time to introduce some terminology.
\begin{df}Fix a deformation, denoted $f$.\begin{itemize}
	\item A path component $\Delta\subset S_+(f)$ is \emph{admissible} if it satisfies the following:\begin{enumerate}\item\label{1} The 1-manifold $\partial\overline{\Delta}$ consists entirely of definite cusp points,\item\label{2} The surface $\overline{\Delta}$ is diffeomorphic to the unit disk $D\subset\C$, sending\begin{center}$\overline{\Delta}\cap M_t\mapsto\{z\in D:\Rp(z)=t\}$,\end{center}\item\label{3} The restriction $f|_{\overline{\Delta}}$ is an embedding.\end{enumerate}
	\item A deformation is admissible when each path component of its definite locus is admissible.
	\item For real numbers $0<a<b<1$, a properly embedded arc $\phi:[a,b]\rightarrow\overline{S_+(f)}$ which is monotonic in the sense that $\phi(s)\in M_s$ for all $s\in[a,b]$ is called a \emph{forward arc}. Using a small perturbation of $\lambda$ if necessary, the endpoints of any forward arc are assumed to lie at cusp points at which $\crit(f)$ is not tangent to $M_a$ and $M_b$.\end{itemize}\end{df}
\begin{subsubsection}{Standard neighborhoods}In this paragraph, fix a deformation, denoted $f$, and suppose $\alpha$ is forward arc for $f$, and choose some small $\epsilon>0$. Because of the local models for the cusp and fold points that constitute $\alpha$, there is a neighborhood $\nu\alpha\subset M_I$ with corners in which the critical locus looks like Figure \ref{stdnbhd1}.
\begin{figure}[ht]\centering\begin{minipage}[t]{0.5\linewidth}\centering\includegraphics[height=1.3in]{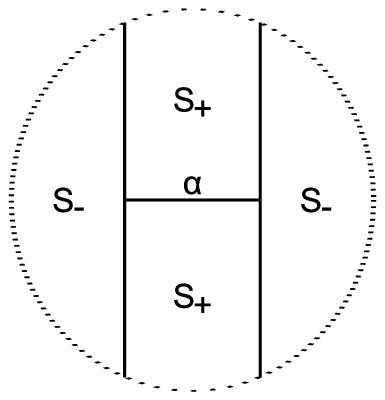}\caption{The critical locus of a standard neighborhood $\tau$.}\label{stdnbhd1}\end{minipage}%
\begin{minipage}[t]{0.5\linewidth}\centering\includegraphics[width=\linewidth]{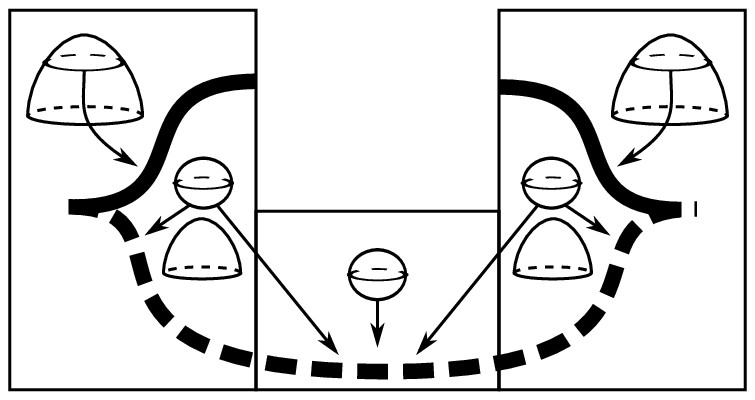}\caption{Base diagram of a cross-section of a standard neighborhood.}\label{stdnbhd2}\end{minipage}\end{figure}
More precisely, a standard neighborhood is such that $\nu\alpha\cap S(f)$ is a disk whose stratification can be given by coloring the unit disk $D\subset\C$ such that the following correspondences hold:\begin{displaymath}\begin{array}{rcl}\alpha &\leftrightarrow & \{z\in \R:|Re\ z|\leq\epsilon\}\\S_+(f)\cap\nu\alpha &\leftrightarrow &\{z\in D:|Re\ z|<\epsilon\} \\S_-(f)\cap\nu\alpha &\leftrightarrow & \{z\in D:|Re\ z|>\epsilon\}\\S^{(3,1,0)}(f)\cap\nu\alpha &\leftrightarrow &\{z\in D:|Re\ z|=\epsilon\}.\end{array}\end{displaymath}A standard neighborhood immediately inherits a fibration structure $\tau:D^4_{[-\epsilon,\epsilon]}\rightarrow D^2_{[-\epsilon,\epsilon]}$ from the local models, which is a perturbed version of a trivial deformation of the map depicted in Figure \ref{stdnbhd2}. As $\nu\alpha$ always comes equipped with the map $\tau$, the notation $\tau$ will denote a standard neighborhood both as a topological space and as a map. The map of Figure \ref{stdnbhd2} is the union of the standard local model for a definite fold arc and two copies of the model for definite cusps, and is divided into three pieces, each corresponding to a local model. The diffeomorphism type of the fiber is labeled, showing the total space as two 0-handles connected to each other by a 1-handle, each piece containing one properly embedded arc of critical points which smoothly meets with that of any adjacent pieces. Another way to understand Figure \ref{stdnbhd2} is as the result of removing a 4-ball fibered as in Figure \ref{fold} from the 4-ball fibered as in Figure \ref{defbirth}; in particular, the common boundary is fibered by disks, so that it is clear that, as in the case of the Cutting Lemma \ref{cut}, any modifications that occur relative to the disk fibers lying above the boundaries of the boxes in Figure \ref{stdnbhd2} actually take place on the interior of the total space. Denoting the map of Figure \ref{stdnbhd2} by $f_{slice}:D^4\rightarrow D^2$, with appropriate coordinates the standard neighborhood map $\tau$ has the expression $\tau(s,x)=(s,f_{slice}(x))$. Using appropriate coordinates, a forward arc is parametrized as $\phi(s)=(s,m)\in M_s$ for $s\in[a,b]$ and some fixed $m\in M$. Since the endpoints of $\phi$ lie at points where the cusp locus is transverse to each slice $M_s$, without loss of generality $\crit(f)\cap\nu\phi$ is transverse to the slices $M_s$. For this reason, there is a straightforward description of $\tau$ using base diagrams. The initial picture is precisely Figure \ref{defcusp}. As $t$ increases, an abrupt change occurs in which the cusp, the indefinite arc, and all disk fibers are omitted from the picture and all that is left is Figure \ref{deffold}. The remainder of the deformation is the reverse of the previous sentence. In this description, the image of $\phi$ is a point that first appears on the definite cusp and immediately moves into the interior of the definite locus as the deformation progresses, returning to and immediately disappearing at the cusp point after it reappears toward the end of the deformation. For a forward arc, the parameter $t$ appears in Figure \ref{stdnbhd2} parametrizing the left-to-right direction. It may help to visualize $\tau$ as a restriction of the deformation in which a definite birth happens to a trivial fibration $D^2\times D^2\rightarrow D^2$, immediately followed by the inverse of a definite birth.\end{subsubsection}
\begin{subsubsection}{A cutting deformation}The goal of this paragraph is to introduce a map $c:D^5\rightarrow D^3$ which is useful because it affords a way to ``cut" the definite locus of a deformation along a forward arc. Much like a standard neighborhood, this deformation has the property that it can be interpreted in two ways, both as a deformation as pictured and as a deformation when taking the parameter $t$ to point in the left-to-right direction. The description of $c$ uses base diagrams, beginning with Figure \ref{stdnbhd2}, on which there occurs a definite birth to obtain the left side of Figure\begin{figure}[ht]\begin{center}\includegraphics[width=\linewidth]{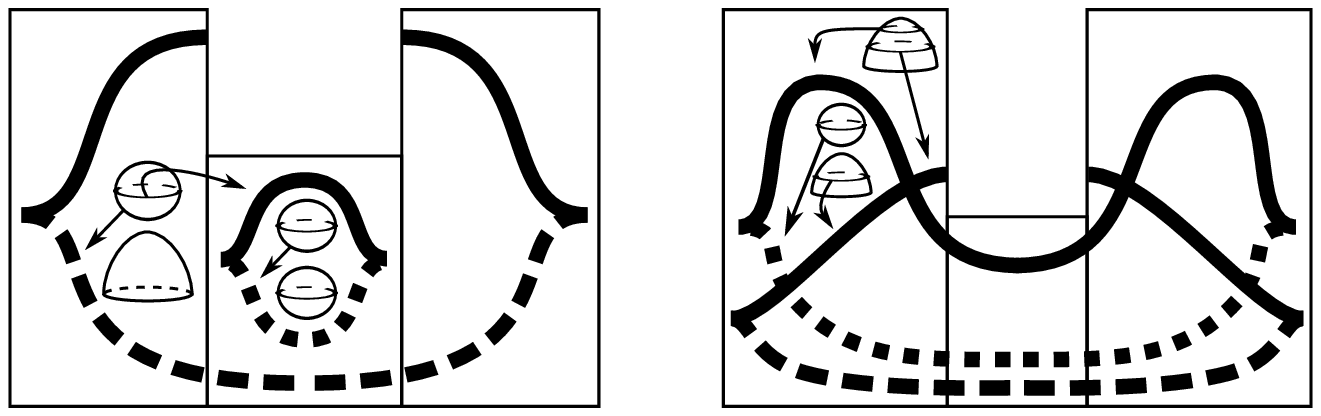}\end{center}\caption{The map $c$.}\label{cut1}\end{figure} \ref{cut1}. Here, the initial placement of the birth requires that the vanishing cycle of the indefinite fold of the new circle lies on the sphere fiber component, and the definite fold arcs are dotted differently to distinguish the two components of $S_+(c)$. An isotopy expands this new circle to obtain Figure \ref{cut1} right, where the two cusps have migrated over the two indefinite arcs, resulting in a new region whose regular fiber is obtained by connect summing one of the two sphere components in the center with the disk component on each side. Another isotopy pushes the inner definite arc outwards, expanding what was originally the inner $S^2$-fibered region to obtain Figure \begin{figure}[ht]\begin{center}\includegraphics[width=\linewidth]{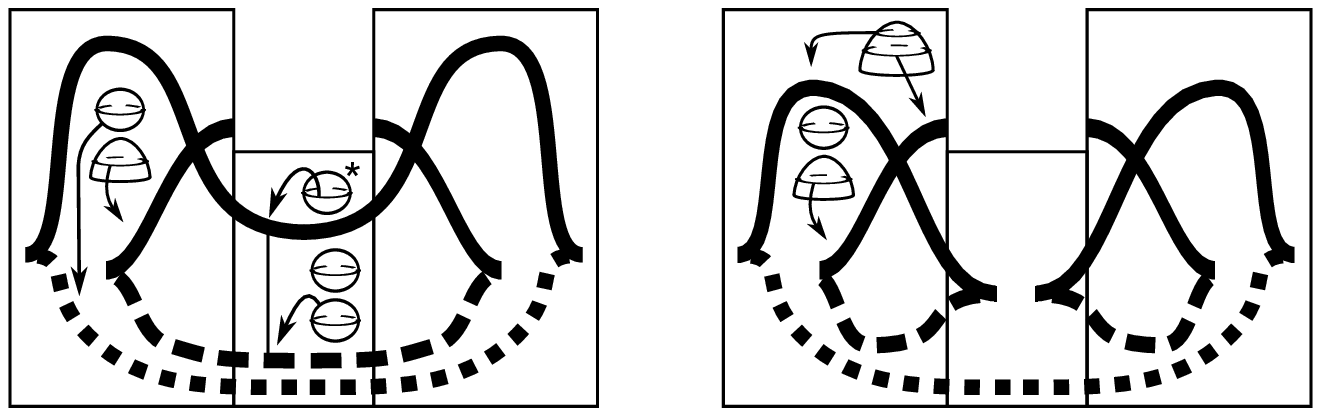}\end{center}\caption{The map $c$.}\label{cut2}\end{figure} \ref{cut2} left. Here a vertical arc signals a definite merge to obtain the right side of the figure (this move is easily seen to be valid by taking a hemisphere of the sphere marked with an asterisk as the disk fiber in Figure \ref{defmerge}). Here, as always, the vanishing cycles are determined by their appearance in the previous base diagram. Comparing the two loops to Figure \ref{idi}, it is clear that they may be removed by reverse IDI-type swallowtails, resulting in a return to Figure \ref{stdnbhd2}.\\

So far, $c$ is a map that does ``cut" along a forward arc, but the two swallowtails are rotated in a way that makes it unclear whether substituting $c$ for a forward arc preserves the deformation condition. An illustration of the image of the critical locus near a swallowtail point (whose image is marked with a dot) appears in Figure \ref{straighten2}, where normally $t$ would be in a direction parallel to the arc of double points; however, when $c$ is substituted for a forward arc, the two swallowtails appear with the homotopy parameter increasing in a perpendicular direction as shown.
\begin{figure}\centering\begin{minipage}[b]{0.5\linewidth}\centering\includegraphics{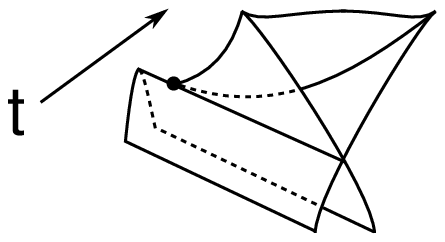}\caption{The critical image of a ``sideways" swallowtail.} \label{straighten2}\end{minipage}\begin{minipage}[b]{0.5\linewidth}\centering\includegraphics{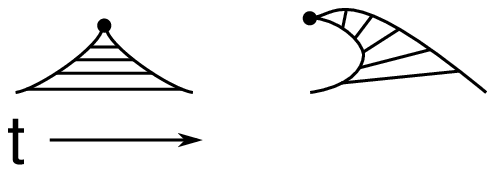}\caption{Twisting the previous picture (here $S_+$ is shaded).} \label{straighten3}\end{minipage}\end{figure} 
A small perturbation near each swallowtail point turns them so they appear with the expected local model. This perturbation is illustrated in Figure \ref{straighten2} by taking the swallowtail point (marked by a dot, where two cusp arcs meet) and pushing it to the left by a smooth homotopy, so that the arc of double points becomes parallel to the $t$ direction momentarily before resuming its original course. It is important to perform this homotopy so that the collection of points at which the cusp locus is tangent to $M_t$ coincides with the collection of points at which the closure of the fold locus is tangent to $M_t$. This condition ensures that these are merge points according to the discussion of Section \ref{criticalmanifold}. The next step is to apply Lemma \ref{defswallowtails} to remove the definite swallowtails, replacing each with a copy of $g_{idi}$. Looking back at $S_+(g_{idi})$, it is evident that the right side of Figure \ref{straighten3} changes so that, within that picture, a small neighborhood of the dot corresponding to the swallowtail point gets replaced with a definite birth point, appearing as in Figure \ref{straighten1}, which is a larger version of the right side of Figure \ref{straighten3} after the application of Lemma \ref{defswallowtails}. The upper part of the figure depicts a neighborhood of a definite cusp arc; the solid arc represents the image of the cusp locus and the dotted arcs represent the boundary of the image of the fold locus, all projected down to a plane.
\begin{figure}[ht]\begin{center}\includegraphics{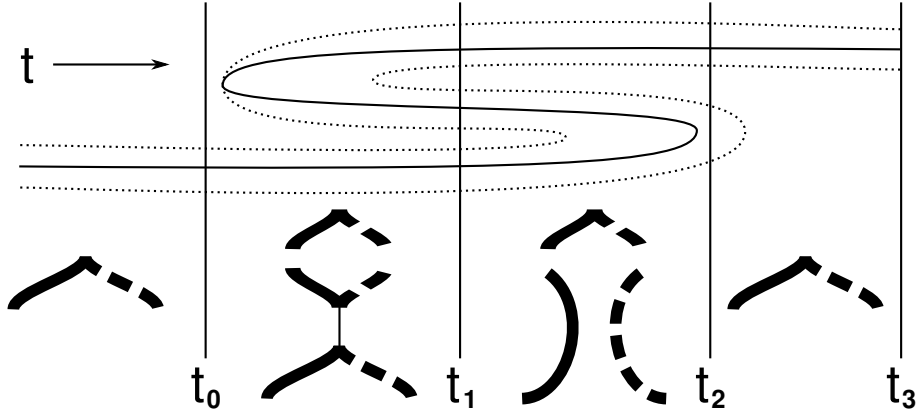}\end{center}\caption{A definite birth point at left that cancels with a definite merge point at right, shown by a dark cusp arc surrounded by a neighborhood in the fold locus.}\label{straighten1}\end{figure}
 Each vertical arc corresponds to a certain value of $t$, and it has a corresponding base diagram to its left. Note that the cusp arc is in general position with respect to $t$, and so its projection to the $t$ axis is morse, with two canceling critical points. The progression begins with a neighborhood of a cusp, followed by a definite birth at the left morse critical point (notice the closure of the fold locus is also tangent to $M_{t_0}$, forming a cup shape that opens to the right). Between $t_0$ and $t_1$ the boundary of the neighborhood cuts into the fold locus at two points, separating the two new cusp arcs to obtain the base diagram corresponding to $t=t_1$, where an arc signals an inverse merge, whose result appears at time $t_2$ (considering the cusps in Figure \ref{straighten2}, the fold locus necessarily forms a saddle near the right morse critical point). Finally two fold arcs disappear from the neighborhood of the cusp arc, giving the diagram for $t=t_3$. Note that at certain times the same cusp arc may appear at three different places in a given slice $M_t$ because of the presence of morse critical points in its $t$-projection. This is problematic because the goal is to make the critical locus admissible and such behavior will violate Condition \ref{2} for admissibility. To this end, note that the modifications between $t_0$ and $t_3$ may be removed by a homotopy involving the cancellation of the left and right morse critical points. Using an appropriate metric on $M$, the maximum lengths of the merge arc and of the fold arcs resulting from the birth approach zero (this is an unstable map, which is a deformation except for the point at which the cusp arc has a removable tangency with a slice $M_t$). The homotopy continues further until it appears that the birth/merge pair never occurred. Considering the discussion of Section \ref{criticalmanifold}, the condition that must be preserved is that no new tangencies should be introduced between the critical locus and the slices $M_t$, which is possible while performing the canceling homotopy relative to the boundary in Figure \ref{straighten1}. Applying this to each pair of points coming from the two modified definite swallowtails, this concludes the description of $c$.
\begin{lemma}[Cutting lemma]\label{cut}Let $\tau:\nu\alpha\rightarrow D^3$ be a standard neighborhood of a forward arc $\alpha$ for a deformation $f$. Then there is a deformation $f'$ such that $f'=f$ outside of $\nu\alpha$, and such that $f'|_{\nu\alpha}=c$.\end{lemma}
\begin{proof}Denote the cross-section map appearing in Figure \ref{stdnbhd2} by $\tilde{k}$, and define a trivial deformation $k(s,x)=(s,\tilde{k}(x))$. The map $\tilde{k}$ is stable, hence there is some product neighborhood on which each slice of $\tau$ is right-left equivalent to $\tilde{k}$. In other words, there are appropriate local coordinates such that there is a subset $U\subset\nu\alpha$ on which $\tau|_U=k$. With this understood, $\tau|_U$ may be depicted with the unchanging base diagram of Figure \ref{stdnbhd2}, with the parameter $t$ understood to parametrize the left-to-right direction. Viewing $c$ and $\tau|_U$ as base diagrams for a pair of deformations, it is clear that the boundary fibrations induced by each are isomorphic, since the modifications taking place in each occur away from the boundary. For this reason $f'$ may be correctly interpreted as the result of a fibered gluing where the fibration induced by $c$ replaces the one induced by $k$, followed by smoothing along the common boundary. For the assertion that $f'$ is a deformation, it is sufficient to check that $\crit(c)$ is of the form that characterizes deformations. At every stage in Figures \ref{cut1} and \ref{cut2}, its description may be interpreted as a deformation when the $t$ parameter is seen moving from left to right; that is, the critical locus is of the appropriate form, as discussed in Section \ref{criticalmanifold}, except for the ``sideways swallowtails" which are addressed above. For the behavior that occurs between the base diagrams in Figures \ref{cut1} and \ref{cut2}, Propositions \ref{birthcusp} and \ref{mergecusp}, along with their definite counterparts, show that the modifications (definite and indefinite births and merges) that take place merely appear as isotopies when $f'$ is viewed as a sequence of base diagrams with increasing $t$.\end{proof}
\begin{figure}[ht]\begin{center}\includegraphics{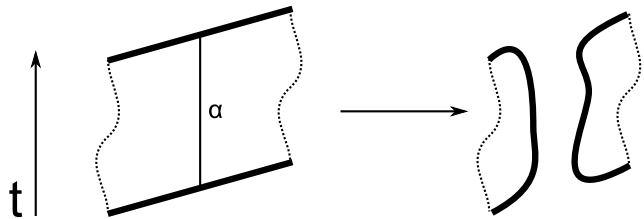}\end{center}\caption{A depiction of how the deformation $c$ modifies the definite locus of a standard neighborhood.}\label{cutpic}\end{figure}In Figure \ref{cutpic} there appears a schematic depiction of how the cutting lemma modifies the definite locus. A strip of definite folds bounded on two sides by cusp arcs may be interpreted as a 1-handle in some handle decomposition of the definite locus in which $\alpha$ is a cocore. The cutting lemma surgers out this 1-handle, producing two cusp arcs. Care has been taken to remove canceling morse critical points in the cusp locus, so that the cutting Lemma produces at most one pair of definite birth or merge points as indicated by two tangencies of the cusp locus with $\{t=const\}$ at the right side of the figure.\end{subsubsection}
\begin{lemma}[Simplification lemma]\label{simplemma}Suppose $f:M_I\rightarrow F_I$ is a deformation. Then there is an admissible deformation $f^{ad}$ such that $f_i=f^{ad}_i$, $i=0,1$.\end{lemma}
\begin{proof}The argument is a repeated application of Lemma \ref{cut}, modifying each path component of $S_+(f)$ to satisfy the conditions of admissibility given in the definition, as follows. 
	\begin{subsubsection}{Condition 1}By Lemma \ref{defswallowtails} and the discussion of Section \ref{criticalmanifold}, without loss of generality $S_+(f)$ is an oriented embedded surface in $M_{(0,1)}$, equal to its interior, such that each path component $C$ satisfies $\emptyset\neq\partial\overline{C}\subset\gamma$, where $\gamma$ denotes the collection of definite cusp points of $f$; thus $S_+(f)$ satisfies Condition \ref{1}.\end{subsubsection} 
	\begin{subsubsection}{Condition 2}Fixing such $C$, there is a handle decomposition as the interior of a 2-dimensional $\{0,1\}$-handlebody. The next step is to show that there is a handle decomposition of $C$ such that each 1-handle $h$ contains a cocore which is also a forward arc; for in that case an application of the cutting lemma to each arc would transform $C$ into a union of disks. Recall that the projection $T:\crit(f)\rightarrow I$ is morse, and note that $\crit(T|_C)=\emptyset$ because of the local model for definite folds. The restriction $T|_\gamma$ is also morse, and as discussed above $\crit(T|_\gamma)=\crit(T|_{\overline{S_+}})$, and these points correspond to definite births, definite merges, and their inverses. Choose a definite birth point $p$ and consider the definite fold arc that it produces. Suppose that after some interval in $t$ the arc eventually takes part in a definite merge at a point $q$. Choose an arc that connects $p$ to $q$ in the region of definite folds under consideration, properly embedded, such that $\partial/\partial t>0$. Perturb so that its endpoints lie within the cusp locus away from $p$ and $q$ to make sure it is a forward arc and apply the cutting lemma. The original merge point persists, but now it is adjacent to a canceling element of $\crit(T|_\gamma)$ and may be removed as in the end of the description of $c$, straightening the newly produced cusp arcs. Now, in order for a definite merge point to exist, there has to be a preexisting arc of definite fold points, which can only arise via definite birth in the absence of definite swallowtails. For this reason, it is possible to apply this process repeatedly until all merge points are removed, possibly increasing the number of definite birth points and components of $S_+$. In a similar way, reversing the parameter $t$, all the inverse definite merge points can be eliminated, possibly increasing the number of inverse definite birth points and components of $S_+$. Calling this new deformation $f'$, and its definite cusp locus $\gamma'$, $\crit(T|_{\gamma'})$ now consists of definite birth and inverse definite birth points. In particular, each component $C'\subset S_+(f')$ appears in a definite birth, and the diffeomorphism type of its cross section $C'_t$ is unchanged until it encounters the next element of $\crit(T|_{\partial C'})$, which is necessarily an inverse birth, at which time the cross section (and $C'$) vanishes. This gives the parametrization condition $\overline{\Delta}\cap M_t\mapsto\{z\in D:\Rp(z)=t\}$ and thus Condition \ref{2}.\end{subsubsection} 
	\begin{subsubsection}{Condition 3}The final step is to cut the disks $\Delta_1,\ldots,\Delta_k$ that constitute $S_+(f')$ (each of which is in general immersed by $f'$) into smaller disks, each of which is embedded under the deformation, which will then be admissible. Fix one component $\Delta_i$ and call its immersion locus $D$. By Lemma \ref{r1}, the tangle corresponding to the Reidemeister-1 move is absent from the tangle depiction of $D$, implying that $D$ is a union of the tangles corresponding to the other two Reidemeister moves. Observing Figure \ref{reid}, any pair of like-numbered components in these two Reidemeister pictures can easily be separated by a vertical arc, and joining these arcs in $\Delta$ results in a forward arc, which causes $A$ and $B$ to lie in different components of $S_+$ after applying the cutting lemma. This gives Condition \ref{3}.\end{subsubsection}\end{proof}\end{subsection}
\begin{subsection}{Removal of the definite locus}\label{removal}In this section fix an admissible deformation $f$ and a path component $\Delta\subset S_+(f)$. The circle $\partial\overline{\Delta}$ consists of cusp points, and so by the local model of the definite cusp there is a tubular neighborhood $\nu=\nu(\overline{\Delta})\subset S(f)$ on which the deformation map is a homeomorphism. For this reason $\nu\subset M_I$ is diffeomorphic (and $f(\nu)$ is homeomorphic) to a complex disk in which $\Delta$ itself is identified with $\{|z|<1\}\subset\C$, the boundary of $\Delta$ consists of a circle of cusps identified with $\{|z|=1\}$, and $\overline{\Delta}$ has a collar consisting of indefinite fold points which is identified with $\{1<|z|<2\}\subset\C$, all of which are embedded under the deformation map. As remarked in Section \ref{simp}, as $t$ increases $\Delta$ must arise via a definite birth as in Figure \ref{defbirth}, and by the same reasoning, it is straightforward to see that $\Delta$ must also vanish via the definite birth model, and between these values of $t$ the restriction $f|_\nu$ has a base diagram given by Figure \ref{stdnbhd2}.

Beginning with a region swept out by an arc of indefinite fold points, one may perform a flip and then remove the loop that results by an inverse flip, resulting in a new deformation that has the addition of two arcs of cusp points which meet at two swallowtail points. Certainly one may extend the second swallowtail (and thus the cusp arcs) forward in $t$ as far as the surface of indefinite fold points extends, and backwards similarly. Following the suggestive Figure 36(1) of \cite{AW}, it makes sense to refer to such a pair of swallowtails as a \emph{canceling pair}. 

With this understood, the first step for removing $\Delta$ is to introduce two pairs of canceling flips alongside $\Delta$, and at this point it is instructive to refer to Figure \ref{cs}, in which the progression is from left to right and back again, beginning and ending with the empty diagram, which describes a trivial fibration by disks. The first step is obviously a definite birth. For the two pairs of canceling swallowtails, the initial two flips occur on the indefinite arc that appeared with the definite birth, and the canceling pairs extend forward with respect to $t$ such the resulting cusped loops in the base diagram persist into the intermediate stage pictured at the right, continuing until $\nu\Delta$ again appears as in the middle diagram. Then the inverse swallowtails appear just before the inverse definite birth, closing off the two loops before the rest of $\crit(\nu\Delta)$ disappears. In this way, for some closed interval $J\subset I$, Figure \ref{cs} describes a map $M_I\supset D^4_J\rightarrow D^2_J$ with two canceling pairs of flips.\\
\begin{figure}[hb]\begin{center}\includegraphics[width=0.8\linewidth]{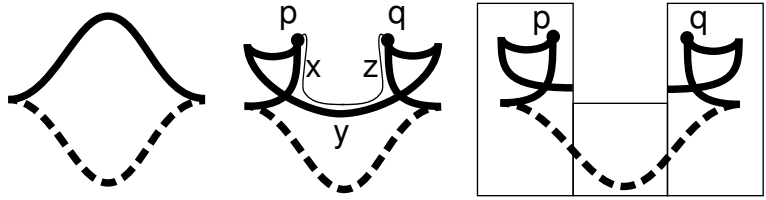}\end{center}\caption{Base diagrams for a neighborhood of an admissible disk $\Delta$ endowed with two pairs of canceling flips and an arc which signals an inverse merging move.}\label{cs}\end{figure}

Consider the two closed arcs $P,Q$ of cusp points, parametrized in $M_I$ by $P_t=(t,P(t))$ and $Q_t=(t,Q(t))$, whose images under $f$ appear in the base diagrams of Figure \ref{cs} as the points $p$ and $q$. Without loss of generality, the initial points of $P$ and $Q$ both lie in $M_a$ and both terminal points lie in $M_b$ for some $[a,b]$ in the interior of $J$. In the middle of Figure \ref{cs}, near to three indefinite arcs labeled $x,y$, and $z$, there appears an arc $\beta_a\subset M_a$ suitable for an inverse merging move between the cusps $P_a$ and $Q_a$ (and in a symmetric fashion there is an arc $\beta_b\subset M_b$). The goal here is to show that there is a one-parameter family of such arcs $\beta_t$, $t\in[a,b]$.\\

The arc $\beta_a$ is a slight perturbation of an arc $\tilde{\beta}_a$ whose image lies entirely within the critical image, and which indeed lies entirely in the indefinite locus except for two short paths in fibers above two points. In the base diagram, this arc runs from $p$ downward along $x$ toward the intersection with $y$, and at that point in $M_a$ it leaves the critical locus, following a path in the fiber above $x\cap y$ to the node corresponding to $y$. The path then proceeds along $y$ into its interior. Coming from $q$ the path proceeds analogously so the two pieces meet in the middle of $y$. Perturbing this arc into the regular locus in the direction away from the sphere fibers gives the smooth arc $\beta_a$ (and analogously $\beta_b$).\\

In the same diagram, the spherical fibers are vanishing cycles for the definite fold arc. Flowing these spherical fiber components as far upward as possible shows that they map to a region which is bounded by a circle consisting of the definite image and three sub-arcs of $x$, $y$, and $z$ (to be precise, this is the region of regular values in the base diagram containing the letter $y$). Appropriately replacing the sub-arcs which are adjacent to the definite cusps gives $\tilde{\beta}_a$. This behavior persists for each $t\in[a,b]$ in the sense that, within each slice $M_t$, flowing the spherical vanishing cycles of $\Delta_t$ outward as far as possible always terminates at some indefinite arc (such a flow must terminate, and if one of these flows were to terminate at a definite arc not contained in $\Delta$, then there must have been a definite merge, which is a contradiction). This singles out a family of arcs $\tilde{\beta}_t$ which sweep out a continuously embedded disk $\tilde{\beta}$, which when perturbed in the direction of flow gives a disk $\beta$ such that for dimensional reasons may be assumed to intersect the critical locus at precisely $P$ and $Q$. For this reason, the slices $\beta_t$, $t\in[a,b]$ form a family of arcs which are then suitable for inverse merge between their cusp endpoints. Thus a 1-parameter family of inverse merging moves may be performed between $P$ and $Q$, along the arcs $\beta_t$, by a homotopy of $f$.\\

As in Example \ref{birthmerge}, the fibration depicted in the center of Figure \ref{cs} can also be obtained from the trivial fibration by disks by performing the definite birth as before, but then performing an indefinite birth followed by an isotopy and a merging move instead of the two flips. This is substituted into the deformation in the same way as in the proofs of Lemmas \ref{defswallowtails} and \ref{cut} at both ends of the modified neighborhood $\nu$. The result is that the component of the critical locus containing $\Delta$ is a sphere, on which the deformation is injective, composed of an indefinite disk glued along its cusp boundary circle to a definite disk, which by the structure of deformations must occur as a definite birth followed by isotopy, ending with an inverse definite birth. Such a sphere $S$ may be removed by a homotopy in which each circle given by $S\cap M_t$ shrinks to a point and disappears via inverse definite birth, the definite birth points at either end approaching each other and disappearing with $S$. In this way, each component of $S_+$ may be removed, and by Theorem 4.4 of \cite{L}, the resulting deformation of wrinkled fibrations is given by a sequence of the moves of Section \ref{icat}. This completes the proof of Theorem \ref{T}.\end{subsection}\end{section}

\begin{section}{Applications and conjectures}\label{apps}\begin{subsubsection}{Simplified purely wrinkled fibrations}\label{spwf}The first existence result for broken Lefschetz fibrations appeared in \cite{ADK} in the case where $M$ is equipped with a \emph{near-symplectic structure}, which is a closed 2-form that vanishes transversely on a smoothly embedded 1-submanifold. More general results followed, beginning with \cite{GK} and culminating with \cite{B2}, \cite{L} (written from the perspective of singularity theory) and \cite{AK} (written using a handlebody argument). After some more terminology, there follows another somewhat more specialized existence result.
\begin{df}Suppose two smooth maps $f_,g$ from a fixed 4-manifold $M$ into a surface $F$ are related to each other by a sequence of the moves of Section \ref{icat}. Then $f$ is \emph{equivalent} to $g$.\end{df}
In particular, any pair of homotopic broken Lefschetz fibrations of $M$ are equivalent. There are homotopic maps which are not equivalent: a simple example is a broken Lefschetz fibration and the new map obtained by performing a definite birth.
\begin{df}Suppose $f:M\rightarrow S^2$ is a stable map such that $S=\crit f$ is a single cusped circle of indefinite folds, such that $f|_S$ is injective. Combining terminology from \cite{B1} and \cite{L}, such a map is called a \emph{simplified purely wrinkled fibration}, or SPWF for short.\end{df}
\begin{cor}\label{spwfcor}Every broken Lefschetz fibration is equivalent to some simplified purely wrinkled fibration.\end{cor}
\begin{proof}In \cite{GK}, the authors show that each smooth oriented 4-manifold admits an \emph{achiral} broken Lefschetz fibration, where the term \emph{achiral} signifies that in complex coordinates their local model reads $(z,w)\mapsto z\overline{w}$ instead of the (complex orientation preserving) $zw$. The map is such that all fibers are connected, the indefinite fold locus is mapped mapped diffeomorphically to a collection of embedded circles parallel to and disjoint from some other embedded circle called the equator, and traveling away from the equator each indefinite circle is oriented such that passing over that circle results in the genus of the fiber decreasing by 1, so that the highest-genus region is a neighborhood of the equator and the lowest-genus region is one or both of the ``poles." The remainder of the critical locus is a finite number of Lefschetz and achiral Lefschetz critical points, all of which appear along the equator. \begin{figure}[ht]\begin{center}\includegraphics[width=\linewidth]{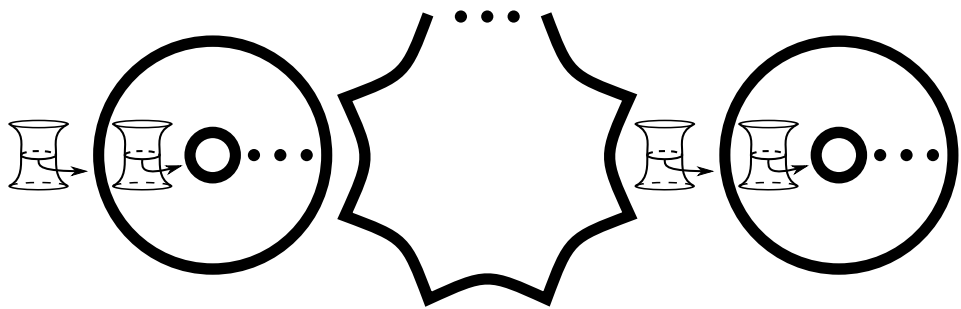}\end{center}\caption{A base diagram for the Gay-Kirby fibration, modified by wrinkling all isolated critical points, then inverse merging the resulting cusped circles.}\label{spwf1}\end{figure} The construction of \cite{GK} begins with an arbitrary embedded, smooth, oriented surface $F\subset M$ with trivial normal bundle, together with a map from a neighborhood $\nu F$ of that surface projecting $D^2\times F\rightarrow D^2$ as a fiber of the fibration. The (arbitrary) parametrization of $\nu F$ implicity specifies a framing for that surface, and hence a framed cobordism class of surfaces. The Thom-Pontrjagin construction shows that homotopy classes of maps $M\rightarrow S^2$ are in one-to-one correspondence with framed cobordism classes of surfaces. Thus there is a map as in Figure \ref{spwf1} in every homotopy class of maps $M\rightarrow S^2$. Finally, Theorem \ref{T} implies any broken Lefschetz fibration homotopic to one of these maps is actually equivalent to it. For this reason, it suffices to give an algorithm to show that such a map is equivalent to some SPWF. In what follows, for ease of reading, a circle of indefinite folds (that is, one that is free of cusps) is called \emph{smooth}. In \cite{L} the author shows that there is a wrinkling type of modification of a neighborhood of an achiral Lefschetz critical point, with the result that the vanishing cycles appear in the reverse order from that found in Figure \ref{wrinkle} (see also \cite{B3} for a handlebody argument for fixing achiral points). Performing this modification on all the achiral (and wrinkling the other) Lefschetz critical points, followed by combining the resulting 3-cusped circles via inverse merging moves, results in a wrinkled fibration $M\rightarrow S^2$ shown in Figure \ref{spwf1}, where a regular value on the equator is omitted from the base diagram to allow a suitable projection. At this point, a version of the modification of Figure 5 of \cite{B2} and Figure 11 of \cite{L} is available to the innermost of the nested circles, shown in Figure \ref{flipandslip}.\begin{figure}[ht]\begin{center}\includegraphics[width=\linewidth]{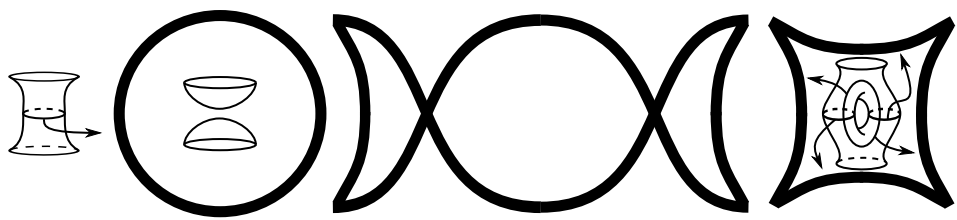}\end{center}\caption{Modifying an indefinite circle by two flips followed by an isotopy.}\label{flipandslip}\end{figure} In general, there will be two nested families of circles as in Figure \ref{spwf1}, and the next step in the modification is to perform the modification of Figure \ref{flipandslip} on the innermost circle of each, then to expand this circle to be the outermost member of its family, and perform an inverse merge between one of the four cusps in that circle and a cusp in the circle that resulted from the wrinkling moves. If there were no isolated critical points to begin with (and hence no cusped circle as in the middle of Figure \ref{spwf1}), then perform an inverse merge between a cusp in each of the the two outermost circles. If there was only one family of nested circles, then inverse merge the cusped outer circle with the one that came from the isolated critical points. The result is a large cusped circle with one or two families of nested smooth circles inside its higher-genus region. Assuming there are two smooth nested families, the modification eliminates one of them by repeating the process of flipping, expanding, and inverse merging until there is only one family of nested smooth circles, which appears in the center of a family of nested cusped circles. If this family of smooth circles is nonempty, perform the modification of Figure \ref{flipandslip} on the innermost circle and expand it to be the outermost circle of its family, repeating until there are no more smooth circles. The result is a map which embeds its critical locus into a base diagram in which the critical image is a collection of $n$ cusped circles nested in $S^2$, all with the same orientation (using the language of \cite{B1}, the fibration is \emph{directed}). Now the circle bounding the lowest-genus region may be modified as in Figure \ref{flipandslip}, reversing its orientation in $S^2$ such that an inverse merge may take place between one of its cusps and one of the cusps of the adjacent circle, which results in a nested family with $n-1$ components. Repeating this process until a unique circle remains results in an SPWF.\end{proof}
Considering the unsinking move in Section \ref{icat}, this corollary implies there is some control over how the round vanishing cycles relate to the Lefschetz vanishing cycles in a broken fibration, since a Lefschetz critical point that results from unsinking a cusp always has a vanishing cycle that intersects that of the nearby indefinite fold transversely at a unique point. Corollary \ref{spwfcor} also implies the following concerning a new way to express a 4-manifold.
\begin{cor}Any smooth closed 4-manifold may be specified by a chain $\{\gamma_i\}_{i\in\Z/k\Z}$ of simple closed curves in an orientable closed surface $F$ such that $\gamma_i$ transversely intersects $\gamma_{i+1}$ at a unique point.\end{cor}
\begin{proof}Here, $k$ is the number of cusps, $F$ is the higher-genus fiber, and the curves $\gamma_i$ are the round vanishing cycles obtained by traveling around the critical image of an SPWF of the 4-manifold. The modification appearing in Figure \ref{flipandslip} and \cite{L,B2} can be used to globally inflate the fiber genus to an arbitrarily high integer, in this way avoiding any ambiguity in gluing data coming from the noncontractible loops in the diffeomorphism groups of $S^2$ or $T^2$.\end{proof}\end{subsubsection}
\begin{subsubsection}{Fibrations in distinct homotopy classes}\label{projection}It is known that the homotopy classes of smooth maps from a 4-manifold to the 2-sphere are in bijective correspondence with framed cobordism classes of smoothly embedded oriented surfaces. For this reason, there are generally infinitely many broken Lefschetz fibrations $M\rightarrow S^2$ which are not related by the moves of Section \ref{icat}. Thus a true uniqueness result for broken fibrations requires a move that relates maps in distinct homotopy classes. The idea is to compose any map $M\rightarrow S^2$ with the projection $S^2\rightarrow D^2$ in a way that is stable, reversible and explicit. Then a short argument shows that all such maps are related by Lekili moves. The first step is to describe a map $\beta_g$ which is used in the construction. Consider the standard handle decomposition of the orientable genus $g$ surface $\Sigma_g$ as the union of a 0-handle with $2g$ 1-handles and a 2-handle. Crossing with an interval, one obtains the analogous handle decomposition, which is induced by a morse function \begin{equation}\tilde{\beta_g}:\Sigma_g\times[0,1]\rightarrow[1/2,1]\end{equation} which is chosen such that the fiber over 1 is a point that expands into a three-dimensional 0-handle as the values of $\phi$ decrease. Continuing in this direction there are $2g$ 1-handles $h_1,\ldots,h_{2g}$, and finally there is a 2-handle attached along a curve that travels over the 1-handles in the order $h_1,h_2,h_1^{-1},h_2^{-1},\ldots,h_{2g}^{-1}$, causing the fiber over $1/2$ to be $\Sigma_g\sqcup\Sigma_g$. With this understood, the map $\beta_g$ is obtained by crossing $\tilde{\beta_g}$ with the circle:\begin{equation}\beta_g=:\tilde{\beta_g}\times id:\Sigma_g\times[0,1]\times S^1\rightarrow[1/2,1]\times S^1.\end{equation}Referring to the local models for definite and indefinite folds, it is clear that $\beta_g$ is a stable map; see Figure \ref{beta} for a base diagram.\begin{figure}[ht]\begin{center}\includegraphics[width=\linewidth]{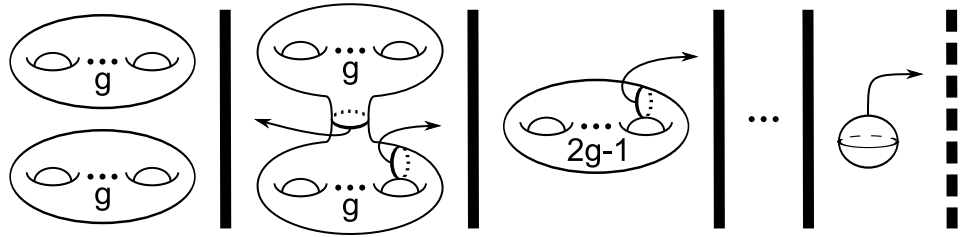}\end{center}\caption{Part of the map $\beta_g$, whose critical image consists of $2g+1$ concentric indefinite circles within a definite circle.}\label{beta}\end{figure}

Suppose a smooth map $f:M\rightarrow S^2$ has a disk $\Delta$ of regular values, with fiber $\Sigma_g$. Taking $\partial\Delta\subset S^2$ as the equator and $p:S^2\rightarrow D^2$ as the obvious projection map that sends $\partial\Delta\mapsto\{|z|=1\}$, the projection move may be taken as a stable unfolding of the map $p\circ f$; however there is a straightforward description by a cut-and-paste operation as follows. The complement of $f^{-1}(\nu\partial\Delta)$ in $M$ is a pair of maps to the disk, one of which is a submersion (that coming from $\Delta$), while the other may have nonempty critical locus. Each of these maps has boundary $\Sigma_g\times S^1$ with the obvious fibration structure. These fibrations are superimposed such that the preimage orientation of the fiber coming from $\Delta$ is reversed and that coming from the other disk is preserved to form a new map \begin{equation}M\setminus(\Sigma_g\times[0,1]\times S^1)\rightarrow\{z\in\C:|z|\leq1/2\}\end{equation} with disconnected fibers; in particular, the preimage of each boundary point is $\Sigma_g\sqcup\Sigma_g$, and gluing $\beta_g$ to this fibration along the boundary in the trivial way completes the map $M\rightarrow D^2$. The projection move is simply to pass from $f$ to this new map, and result of applying this move to a broken Lefschetz fibration (or SPWF, purely wrinkled fibration, stable map, etc.) will be called a \emph{projected} broken Lefschetz fibration (respectively, projected SPWF, etc.).\\

\begin{lemma}\label{cylinder}Suppose $f,g:M\rightarrow D^2$ are projected purely wrinkled fibrations for some closed, smooth 4-manifold $M$. Then there is a deformation between them that is realized as a sequence of the moves in Section \ref{icat}.\end{lemma}
\begin{proof}Certainly $f$ is homotopic to $g$ because $D^2$ is contractible, so begin with an arbitrary deformation between $f$ and $g$. Any stable map $M\rightarrow D^2$ has a compact image whose boundary is composed of definite fold points \cite{S}, and in order to prove the lemma, it is sufficient to arrange that the definite circles $S_+(f)\subset M_0$ and $S_+(g)\subset M_1$ bound a cylinder $C\subset M_{[0,1]}$ consisting entirely of definite fold points, on which the deformation is injective, and further that $C_t\subset M_t$ maps to the boundary of the image for all $t$. Then for a sufficiently small neighborhood $\nu C$, the restriction to $M_{[0,1]}\setminus\nu C$ is a deformation between purely wrinkled fibrations over the disk relative to $\partial M$. Then Theorem \ref{T} implies the remaining definite locus may be removed, resulting in the required deformation $h$. It remains to construct $C$.\\

Beginning with an arbitrary deformation $h'$ such that $h'_0=f$ and $h'_1=g$, for dimensional reasons it is possible to choose an arc $\gamma:[0,1]\rightarrow M_{[0,1]}$ such that $\gamma(t)$ is a regular point of $h'_t$ for all $t\in[0,1]$. Further, there is a neighborhood $\nu\gamma$ which is also disjoint from $\crit(h')$. Finally, one may choose $\gamma$ so that its endpoints are sufficiently close to $S_+(h'_0)$ and $S_+(h'_1)$ to lie on the spherical fibers coming from the definite local models. By a homotopy of $h'$ which is supported on $\nu\gamma$, the first step is to perform a definite birth centered at each point of $\gamma$; in $M_{[0,1]}$ the new critical locus appears as a cylinder parallel to $\gamma$, with a decomposition into a rectangle of definite fold points meeting a rectangle of indefinite fold points along two arcs of definite cusp points. In each slice $M_t$, a small circle appears as in Figure \ref{defbirth}\begin{figure}\begin{center}\includegraphics[width=\linewidth]{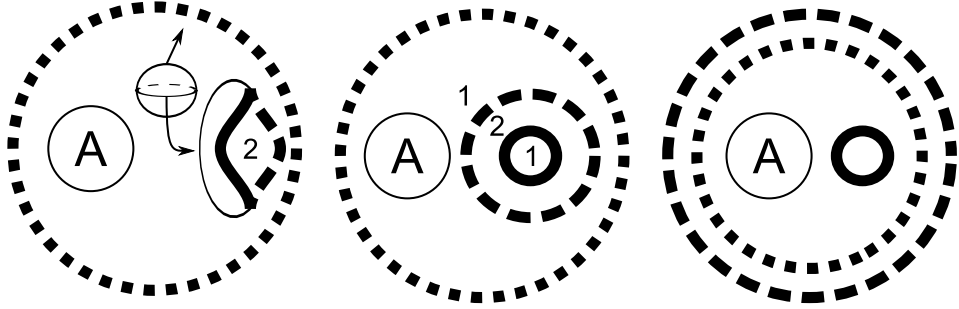}\end{center}\caption{Modifications of a deformation $h$ as they appear at $t=0$. Integers indicate the number of sphere components in the fiber, which is a disjoint union of spheres in all regions except possibly $A$. The initial point of $h''$ is on the right.}\label{cylfig1}\end{figure}; the result of this homotopy in the slice $M_0$ appears in the left side of Figure \ref{cylfig1}, where the region marked $A$ is some purely wrinkled fibration coming from the projection move. The next step is to perform, at every value of $t$, inverse merging moves between the definite cusp points as indicated by the left side of the figure, so that the definite circle is on the outside and the indefinite circle is on the inside for all $t$ (the middle of the figure shows what this looks like at the slice $M_0$). At this point, the deformation $h'$ has been altered such that the endpoints are not identical to $f$ and $g$, and two new circles of critical points sweep out a pair of cylinders parallel to $\gamma$, one definite (which we denote $\tilde{C}$, and which appears as the more heavily dotted circle in the middle of Figure \ref{cylfig1}) and the other indefinite.\\

The sphere fibers that contract at $\tilde{C}$ map to an annulus in $D^2$ at each value of $t$, depicted as the region labeled with a 2 in the middle of Figure \ref{cylfig1} at $t=0$. At each value of $t$, the definite circle forming the outer edge of this annulus may be expanded by a homotopy of $h'$ such that it becomes the boundary of the image at each $t$. The result of this expansion at $t=0$ appears in the figure at the right, and this new deformation will be called $h''$. By construction, $h''$ now has an outermost definite cylinder which up until now has been called $\tilde{C}$, and which now satisfies the requirements to be $C$, except it remains to arrange for the endpoints of $h''$ to be $f$ and $g$.\\

Figure \ref{cylfig2}\begin{figure}\begin{center}\includegraphics[width=\linewidth]{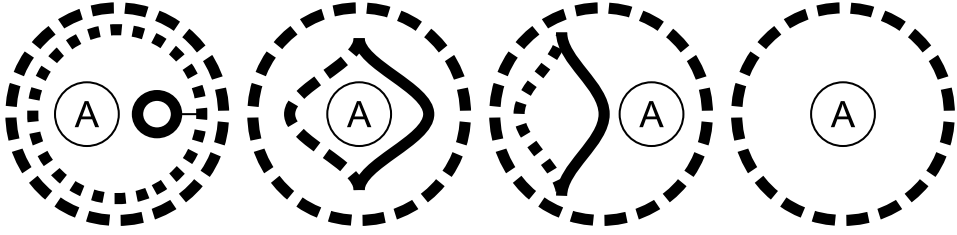}\end{center}\caption{A deformation that serves as a prefix to $h''$.}\label{cylfig2}\end{figure} describes a deformation that begins with $h''_0$ and ends with a map that is related to $f$ by an isotopy. Reversing the parameter $t$ and appending this map to the beginning of $h''$ (and the analogous counterpart to the end of $h''$) gives the required deformation $h$. The list of moves is straightforward: from left to right is a definite merge, then an isotopy in which the definite arc passes across the region $A$, and finally an inverse definite birth. Unlike the other two steps, the isotopy is not obvious and requires further argument in which we consider $A$ as the complex unit disk. Just before the isotopy in question, the fibration over $A$ has the following characterization. The cusped circle around $A$ contributes to each point preimage a sphere fiber component. Aside from these extra spheres, referring to Figure \ref{beta}, a neighborhood of $|z|=1$ has its own sphere fibers which with decreasing radius increase in genus to some even integer $2g$. Traveling further, the fiber separates into a pair of genus $g$ surfaces at an oppositely oriented indefinite circle. At this radius, the point preimage is $\Sigma_g\sqcup\Sigma_g\sqcup S^2$. Because of the way the projection move is defined, one of these genus $g$ fiber components (coming from the disk $\Delta$) traces a trivial fibration over the remainder of $A$, while the other may be part of a more complicated, yet disjoint, fibration structure. The isotopy now explicitly makes sense as one that takes place in the complement of the more complicated piece, resulting in a fiberwise connect sum between the sphere component and the appropriate fiber component as the indefinite arc sweeps from right to left.\end{proof}
\begin{thm}\label{T2}Suppose $M$ is a smooth 4-manifold, and that $f:M\rightarrow S^2$ is a stable map or a broken Lefschetz fibration. Then for any stable map or broken Lefschetz fibration $g:M\rightarrow S^2$, there is a sequence of moves of Section \ref{icat}, along with the projection move, relating $f$ to $g$.\end{thm}
\begin{proof}Theorem \ref{T} covers the case where $f$ and $g$ are homotopic. When $f$ and $g$ are not homotopic, the first step is to apply the projection move to $f$ and $g$. Then Lemma \ref{cylinder} implies the moves of Section \ref{icat} are sufficient to prove the theorem in the case that $M$ is closed. If $\partial M\neq\emptyset$, then more complicated phenomena may arise; however one may simply double $M$ (as well as $f$ and $g$ in the obvious way), apply Lemma \ref{cylinder}, and restrict the result to $M\subset DM$ to obtain the required sequence of moves.
\end{proof}
\end{subsubsection}
\end{section}
\begin{section}{Questions}\label{appendix}
This section lists some related questions of interest to this author and a few of the issues surrounding this work that still require a satisfactory treatment.
\begin{subsubsection}{Minimal genus}The result of section \ref{spwf} raises various minimal genus questions, connected with the obvious uniqueness question for SPWF:
\begin{df}Let $F\subset M$ be a smoothly embedded surface with trivial normal bundle. \begin{itemize}\item The \emph{broken genus} $g_b(M,\alpha)$ of a 4-manifold is the minimal genus fiber of all simplified broken Lefschetz fibrations $M\rightarrow S^2$ whose fiber is in the homology class $\alpha\in H_2(M)$.\item The wrinkled genus $g_w(M,\alpha)$ is the analogous invariant for SPWF.\end{itemize}\end{df}
Thus, along with the classical minimal genus, there are a few distinct notions of minimal genus for fibration structures. The invariants $g_b$ and $g_w$ are not a priori the same, if only because of the more complicated relation between Lefschetz and round vanishing cycles in broken fibrations. Given a fibration realizing $g_b$, it is not difficult to obtain some SPWF of genus $g_b+1$, and the wrinkling move gives a way to turn any SPWF into a broken fibration of the same genus, so it is possible that these invariants are interchangeable. Certainly symplectic Lefschetz fibrations realize the smallest fiber genus among all these notions, so it may be that $g_b$ and $g_w$ offer some kind of measure of ``how far" a manifold is from being symplectic.\end{subsubsection}
\begin{subsubsection}{Uniqueness of SPWF}The following questions seem likely to have affirmative answers:
\begin{conj}Any two homotopic simplified broken Lefschetz fibrations are related by a sequence of modifications shown in Figure \ref{flipandslip}, called ``flip and slip" in \cite{B2}.\end{conj}
\begin{conj}Let $f_1,f_2:M\rightarrow S^2$ be simplified broken Lefschetz fibrations obtained from the SPWF $f:M\rightarrow S^2$ by performing flip-and-slip modifications where the initial flips are chosen to occur at different points of $\crit{f}$, then unsinking all cusps. Then $f_1$ and $f_2$ are related by performing Hurwitz moves on their Lefschetz critical points.\end{conj}
The difficulty in proving either of these conjectures is in keeping track of the round vanishing cycles when the initial flips are performed on distinct components of the fold locus of $f$. When the flips occur on the same component, the result is easy to deduce.\end{subsubsection}
\begin{subsubsection}{Isotopies}This paper uses exclusively \emph{ad hoc} arguments to show that various isotopies are valid, at various points relying on diagrams which are as explicit as possible to convince the reader of the validity of a move, and this can be tedious. A systematic way to show that a given modification is a valid isotopy would make verification of base diagram arguments easier and thus more reliable.\end{subsubsection}
\begin{subsubsection}{Nullhomologous fiber components}The main theorem of this paper states that there exists a deformation between any pair of stable maps which is free of definite folds, assuming the endpoints are free of definite folds. More than simply the existence of \emph{some} such deformation, the proof that Perutz's Lagrangian matching invariant is a diffeomorphism invariant requires the existence of a \emph{near-symplectic cobordism}, in the language of \cite{P1}. This work does not imply such an existence result: it is easy to kill the near-symplectic condition by an isotopy in which a bit of an indefinite fold wanders back over itself, creating a nullhomologous (thus not near-symplectic) sphere component in the fiber. In order to prove diffeomorphism invariance, one might show, for example, that one can arrange for a deformation to have essential fiber components at each stage. Then the modified Donaldson-Gompf construction of a near-symplectic structure may be performed at almost every value of $t$ to obtain a near-symplectic deformation.\end{subsubsection}
\begin{subsubsection}{Removal in more general contexts}It may be interesting to singularity theorists to find the conditions under which one may eliminate definite folds from arbitrary stable maps $X^5\rightarrow Y^3$ by homotopy. It seems feasible to show that the modifications in this paper (specifically, those of the definite swallowtail substitution Lemma \ref{defswallowtails}, and, maybe less easily, the Cutting Lemma \ref{cut}) are realizable as homotopies. Then a short argument that one may open a disk of indefinite fold points within any closed surface of definite folds (using a higher-dimensional version of Example \ref{birthmerge}) would virtually complete the argument.\end{subsubsection}
\end{section}

\end{document}